\documentclass[12pt,psamsfonts]{amsart}

\usepackage{amsmath,amsfonts,amssymb,amsmath,amscd}
\usepackage[active]{srcltx}
\usepackage{cite}
\usepackage{graphicx,array} 
\usepackage{amsthm}
\usepackage{amsbsy}
\usepackage{tikz}
\usepackage{multirow}
\usepackage{float}
\usepackage{hyperref}

\usepackage{tabularx} 

\newcolumntype{b}{X}
\newcolumntype{s}{>{\hsize=.5\hsize}X}

\newcolumntype{L}{>{\displaystyle}l}
\newcolumntype{C}{>{\displaystyle}c}
\newcolumntype{R}{>{\displaystyle}r}

\def \R{\mbox{${\mathbb R}$}} 
\def \C{\mbox{${\mathbb C}$}} 
\def \s{\mbox{\rm{${\textbf{S}}$}}}
\def \r{\mbox{\rm{${\bf \Re}$}}}
\def \D{\mbox{\rm{${\textbf{D}}$}}}
\def \T{\mbox{\rm{${\textbf{T}}$}}}
\def \z{\mbox{\rm{${\textbf{Z}}$}}}

\def \P{\mbox{${\mathcal{P}}$}}
\def \PP{\mbox{${\vec{\mathcal{P}}}$}}
\def \Q{\mbox{${\mathcal{Q}}$}}        
\def \QQ{\mbox{${\vec{\mathcal{Q}}}$}} 

\newcommand{\fix}{\mathrm{Fix}}

\DeclareMathOperator{\im}{Im}
\DeclareMathOperator{\aut}{Aut}
\DeclareMathOperator{\re}{Re}

\newtheorem{theorem}{Theorem}[section]
\newtheorem{lemma}[theorem]{Lemma}
\newtheorem{proposition}[theorem]{Proposition}

\newtheorem{remark}[theorem]{Remark}

\begin{document}

\title{Normal forms of bireversible vector fields}

\author[P. H. Baptistelli, M. Manoel and I.O. Zeli]
{P. H. Baptistelli$^{1}$, M. Manoel$^2$ and I. O. Zeli$^{3}$}

\address{$^1$ Department of Mathematics, UEM, C.P.  5790, 87020-900, Maring\'a-PR, Brazil.} 
\email{phbaptistelli@uem.br}

\address{$^2$ Department of Mathematics, ICMC--USP, C.P. 668, 13560-970 S\~ao Carlos-SP, Brazil.} \email{miriam@icmc.usp.br}

\address{$^3$ Department of Mathematics, IMECC-UNICAMP,  C.P. 651,  13083--859, Campinas-SP, Brazil.}
\email{irisfalkoliv@ime.unicamp.br} 



\maketitle


\begin{abstract}
In this paper we adapt the method of  [P. H. Baptistelli, M. Manoel and I. O. Zeli. Normal form theory for  reversible equivariant vector fields. {\it Bull. Braz. Math. Soc.}, New Series \textbf{47} (2016), no. 3,  935-954] to obtain normal forms of a class of smooth bireversible vector fields. These are  vector fields  reversible under the action of  two  linear involution and whose linearization has a nilpotent part and a semisimple part with purely imaginary eigenvalues. We show that these  can be put formally in normal form preserving the reversing symmetries and their linearization. The approach we use is based on an algebraic structure of the set of this type of vector fields. Although this can lead to extensive calculations in some cases, it is in general a simple and algorithmic way  to compute  the normal forms. We present some examples, which are Hamiltonian systems without resonance for one case and other cases with certain resonances. 
 
\end{abstract}

\section{Introduction}
Many problems in dynamical systems carry special structures to be kept preserved in their systematic qualitative study.  An important such feature is the presence of symmetries acting on  the state variables, implying a time-preserving invariance of the dynamics under this action 
 The particular case of reversing symmetries  also implies an invariance of the dynamics, but in this case with a reversion in time. A simple example is the dynamics of the ideal pendulum (no energy loss), which is reversible with respect to an involution given by the reflection across its vertical axis.  Systems with both symmetries and reversing symmetries,   the so-called reversible equivariant systems, have been studied largely by many authors in a variety of view points (see \cite{manoel1, manoel3, manoel2, BMZ, Buzzi, marco1, Hoveijn,  lima, marco2, Vander}). This paper is a contribution to the local qualitative analysis of bireversible systems defined on a finite dimensional vector space $V$, namely systems in presence of two linear  involutory reversing symmetries $\varphi$ and $\psi$ acting on $V$. An involution is an invertible mapping which is its own inverse. We assume that the two involutions commute,  so the group in action is ${\bf Z}_2 \times {\bf Z}_2$, one component being generated by $\varphi$ and the other by $\psi$. In general,  symmetries and reversing symmetries in a group $\Gamma$ are given in algebraic terms through a group homomorphism $\Gamma \to {\bf Z}_2=\{\pm 1\}$, whose kernel is the subgroup of symmetries, the reversing symmetries being the elements in its complement. 
Hence, in the case under consideration,  ${\bf Z}_2 \times {\bf Z}_2 \to {\bf Z}_2=\{\pm 1\}$ 
is the epimorphism which assumes $-1$ on $\varphi$ and $\psi$, and  1 on the identity element and on the composition $\varphi\circ\psi$.

Our study is based on normal form theory, which is  applied to the system around an equilibrium point, assumed to be the origin. Taking coordinates, this method consists of successive changes of coordinates of the form $I + \xi_k$, for $k \geq 2$, where $I$ is the identity and $\xi_k$ is a homogeneous polynomial of degree $k$. These are to be chosen to put the system in a ``simpler" form at each degree-$k$ level, leaving unchanged the lower-order terms. The vector field obtained is then formally conjugate to the original one, in the sense that their Taylor series are conjugate as 
formal vector fields. It is an interesting fact that the normal form reduction process can introduce additional symmetries into the problem (see \cite[XVI Theorem 5.3]{golub}). In fact, the normal form can be chosen to be equivariant with respect to the one-parameter group given by the closure
\begin{equation}
\s=\overline{\left\{ e^{ sL^t}, s\in \R \right\}},
\label{S}
\end{equation}
where $L$ is the linearization of the vector field at the equilibrium point. As a consequence, if the vector field possesses a group $\Gamma$ of symmetries, then the normal form is $\s \times \Gamma-$equivariant (see \cite[Theorem XVI 5.9]{golub}). In \cite{BMZ} we prove that if the vector field is $\Gamma-$reversible-equivariant, then the truncated normal
form is  $\s \rtimes \Gamma-$reversible-equivariant (Theorem~\ref{thm:NORMALFORM}), and an algorithm is given for the computation of this normal form, based on algebraic invariant theory methods. \\

The aim of this paper is to show that we can adapt the method  developed in \cite{BMZ} for the special case when $\Gamma$ is generated by two commuting involutions when they both act as reversibilities.   The idea follows three steps. We first  obtain an algorithm to compute generators for the module of mappings that are reversible equivariant under a group which is a semi-direct product $\Gamma_1 \rtimes {\bf Z}_2$. This procedure contains an algorithm given  in \cite{manoel1}  as a subroutine (\cite[Algorithm 3.7]{manoel1}), which is applied to $\Gamma_1$ and it is combined with the  construction of a transfer operator to deal with the other component of the whole group. As an intermediate step, elements in $\Gamma_1$ are changed to act as symmetries, which in practice means that we consider, at this stage, a new group homomorphism containing $\Gamma_1$ inside its kernel. Finally, we re-apply the first step to $\Gamma_1 \rtimes ({\bf Z}_2 \times {\bf Z}_2)$ replacing $\Gamma_1$ in the first step  by $\Gamma_1 \rtimes {\bf Z}_2$. It should be clear that we present this method  as a simpler alternative to the algorithm given in \cite{BMZ} for the special class of semi-direct products. 
Some results hold for the more general case when the fist component $\Gamma_1$ of the semi-direct product is any compact Lie group, so we shall present these in this generality.


We look at   ${\bf Z}_2 \times {\bf Z}_2-$reversible-equivariant systems
\[ \dot{x} = X(x) \]
defined on $\R^{2n+2}$, that is, when the vector field $X$ anti-commutes with the group generators,
\[ X \varphi \ = \ - \varphi X, \ \ X \psi \ = \ - \psi X. \] 
We assume that the  linearization of $X$ about the origin has matricial form with a
2-dimensional nilpotent part and a
semisimple part with purely ima\-ginary eigenvalues,
\begin{equation}\label{L}
L = \begin{pmatrix}
  0 & 1 &  &  &  &  &  \\
  0 & 0 &  &  &  &  &  \\
   &  & 0 & \omega_1 &  &  &  \\
   &  & -\omega_1 & 0 & &  &  \\
   &  &  &  &  \ddots & & \\
  &  &  &  &  & 0 & \omega_n \\
   &  &  &  &  & -\omega_n & 0
\end{pmatrix},\end{equation} for nonzero $\omega_i,$ $i=1, \cdots, n.$
Here we deduce the normal forms for the non-resonant case and cases under certain resonances. The main motivation is that all the examples we consider are Hamiltonian systems, for which not so much work has been done due to the difficulty in obtaining normal forms under resonances. The algorithm we present here computes normal forms under any resonance condition, and that is why we have included four diferent cases.
These cases generalizes two examples presented in \cite{BMZ}. The normal form for the nonresonant case which is reversible under a unique involution appears in \cite{lima}, where the authors use the classical method developed by Belitskii (see \cite{belitskii}) in the context without symmetries to compute first a pre normal form, imposing the reversibility afterwards. A comparison between this computation and the method we present shows clearly that the computation is largely simplified if we do not leave symmetries or reversibilities to be taken  into consideration only {\it a posteriori}. Let us remark that if  $\varphi$ and $\psi$ do not commute yet generate a finite group, then this is a dihedral group  $\D_m$ for some $m \geq 3.$ This leads to the invariant theory for the group $(\s \times \z_m) \rtimes \z_2,$ which can also be treated using the results of the present paper. \\  

The paper is organized as follows: In Section~\ref{sec:PRELIMINARIES} we fix our notation,  give some definitions and recall some results. In Section~\ref{sec:SEMIDIRECT}
we present our main results, namely the invariant theory for the groups that are given as a semi-direct product $\Gamma_1 \rtimes {\bf Z}_2$. These are applied to give an algorithm to compute a set of generators for the reversible equivariants under this type of groups. Finally, Section~\ref{sec:NORMAL FORMS} adresses the computation of the normal forms of bireversible systems whose linearization is given in (\ref{L}) under distinct re\-sonance conditions.


\section{Preliminaries} \label{sec:PRELIMINARIES}

Let $\Gamma$ be a compact Lie group  acting linearly on a finite-dimensional real vector space $V$ by $\Gamma \times V \to V,$  $(\gamma, v) \mapsto \gamma v.$ 
In what follows we shall also use the representation $\rho: \Gamma \to {\bf GL}(n)$ associated with this action, namely  $\rho(\gamma)v = \gamma v,$ for all $\gamma \in \Gamma$ and $v \in V.$ 
We also consider a group homomorphism
\begin{equation*}
\label{defisigma}
\sigma: \Gamma \to \z_2,
\end{equation*}
where $\z_2 =\{\pm 1 \}$ is the multiplicative group, 

If $\P_V$ denotes the vector space of polynomial functions $V \to \R$ and $\PP_V$ denotes the vector space of polynomial mappings $V \to V, $ we recall that 
a polynomial function $f \in \P_V$ is called $\Gamma-${\it invariant} if
\[f(\gamma v)=f(v), ~\forall \gamma \in \Gamma, \ v \in V,
\]and it is called $\Gamma-${\it anti-invariant} if
\[
f(\gamma v)=\sigma(\gamma)f(v), ~\forall \gamma \in \Gamma, \ v \in V.
\] A polynomial mapping $g \in \PP_V$ is $\Gamma-${\it equivariant} if
\begin{equation}\nonumber
g(\gamma v)=\gamma g(v),~\forall \gamma \in \Gamma, \ v \in V,
\end{equation}
and it is  $\Gamma-${\it reversible-equivariant} if
\begin{equation}\nonumber
g(\gamma v)=\sigma(\gamma) \gamma g(v),~\forall \gamma \in \Gamma, \ v \in V.
\end{equation}
\noindent Motivated by the nomenclature above, when the homomorphism $\sigma$ is 
nontrivial, then an element $\gamma \in \Gamma$ is called  {\it symmetry} of $\Gamma$ if   $\sigma(\gamma)= 1$, and {\it reversing symmetry}  if  $\sigma(\gamma)=-1$.   We denote by $\Gamma_+ = \ker \sigma$ the subgroup of symmetries of $\Gamma, $ which is a normal subgroup of $\Gamma$ of index $2,$ and $\Gamma = \Gamma_+ \dot{\cup} ~\delta \Gamma_+$ for an arbitrary reversing symmetry $\delta \in \Gamma$. If $\sigma$ is trivial, then all elements in the group are symmetries. 

We shall denote by $\P(\Gamma)$ the ring of the $\Gamma-$invariant polynomial functions, by $\Q_\sigma(\Gamma)$ the module of the $\Gamma-$anti-invariant polynomial functions, by $\PP(\Gamma)$ the module of the $\Gamma$-equivariant polynomial mappings and by $\QQ_\sigma(\Gamma)$ the module of the $\Gamma-$reversible-equivariant polynomial mappings, over the ring $\P(\Gamma)$. The modules $\Q_\sigma(\Gamma),$ $\PP(\Gamma)$ and $\QQ_\sigma(\Gamma)$ are finitely generated and graded over the ring $\P(\Gamma),$ which is also finitely generated and graded. When $\sigma$ is trivial, then $\P(\Gamma)$ and $\Q_\sigma(\Gamma),$ as well as $\PP(\Gamma)$ and $\QQ_\sigma(\Gamma)$, coincide.

In \cite{BMZ},   a method is given to obtain formal normal forms of reversible
equivariant vector fields under the action of $\Gamma$ based on the
classical method of normal forms combined with tools from invariant
theory. More specifically, consider a system of ODEs
\begin{equation}
\label{system1}
\dot{x}= X(x),~ x\in V,
\end{equation}
where $X$ is a $C^\infty$ $\Gamma$-reversible-equivariant vector field. From the linearization $L$ of $X$ at the origin, consider the group $\s$ as defined in (\ref{S}) which acts on $V$ by matrix product. The algebraic method given in \cite{BMZ} consists of computing  the truncated normal form of (\ref{system1}) at any degree via the computation of the generators of the module of homogeneous reversible equivariants under the group $\s \rtimes \Gamma:$ 

 \begin{theorem} (\cite[Theorem 4.7]{BMZ}) \label{thm:NORMALFORM} Let $\Gamma$ be a compact Lie group acting linearly on $V$ and consider $X: V \to V$ a smooth $\Gamma-$reversible-equivariant vector field, $X(0)=0$ and $L=(dX)_0$. Then $(\ref{system1})$ is formally conjugate to
\begin{equation} \nonumber
\dot{x}=Lx + g_2(x) + g_3(x) + \ldots
\end{equation}
where, for each $k \geq 2 $, $g_k$ is a homogeneous of degree $k$ in $\QQ_\sigma(\s \rtimes \Gamma).$
\end{theorem}

We remark that there are cases for which the group $\s$ fails to be compact. Nevertheless,  the tools obtained in \cite{manoel1} and \cite{BMZ} can still be applied as long as the ring $\P(\s)$ and the module $\PP(\s)$ is finitely generated.\\

When $L$ has only purely imaginary eigenvalues we can characterize $\s$ in a particular way. For this, we consider the Jordan-Chevalley decomposition for $L,$ $L = D + N,$ where $D$ is diagonal and $N$ is nilpotent with $DN = ND.$  Then the general form of $\s$ may be deduced as:

\begin{proposition} \label{prop:SFORM} (\cite[Proposition XVI 5.7]{golub}) Let $L = D + N$ be the Jordan-Chevalley decomposition for $L$ and let $k$ be the number of algebraically independent eigenvalues in $D$. If $N=0,$ then $\s = T^k$ and if $N \neq 0,$ then $\s = \r \times T^k,$ where $\r \simeq \bigg\{ \begin{pmatrix}
  1 & 0 \\
  s & 1
\end{pmatrix}: s \in \mathbb{R} \bigg \}$ and $T^k = {\bf S}^1 \times \ldots \times {\bf S}^1$ is the $k-$torus.
\end{proposition}

\section{Invariant theory for the group $\Gamma_1 \rtimes {\bf Z}_2$}  \label{sec:SEMIDIRECT}

Let $\Gamma_1$ and $\Gamma_2$ be compact Lie groups  acting linearly on $V$. Let $\rho: \Gamma_1 \to {\bf GL}(n)$  and : $\eta: \Gamma_2 \to {\bf GL}(n)$ denote  the representations of $\Gamma_1$ and $\Gamma_2$ on $V$, respectively. A semidirect product $\Gamma_1 \rtimes \Gamma_2$ is the direct product  $\Gamma_1 \times \Gamma_2$ as a set with the group operation
$$(\gamma_1,\gamma_2) \cdot_{\mu} (\tau_1,\tau_2) = (\gamma_1 \mu(\gamma_2) \tau_1, \gamma_2 \tau_2)$$
induced by a homomorphism $\mu : \Gamma_2 \to \aut (\Gamma_1)$. In this case $\Gamma_1$ is a normal subgroup of $\Gamma_1 \rtimes \Gamma_2$ and $\Gamma_1 \rtimes \Gamma_2 / \Gamma_1 $ is isomorphic to $\Gamma_2.$ If $\mu$ is trivial, the groups commute
and semidirect is direct product as a group. Now, we define the operation $(\Gamma_1 \rtimes \Gamma_2) \times V \to V$ by \begin{equation}
(\gamma_1,\gamma_2) v =  \rho(\gamma_ 1) (\eta(\gamma_2) v).
\label{acaoproduto}
\end{equation}

From \cite[Proposition 3.1]{BMZ} we have that $(\ref{acaoproduto})$ defines an action of the semidirect product $\Gamma_1\rtimes \Gamma_2$ on $V$ if, and only if,  \begin{equation} \label{action}\rho(\mu(\gamma_2)(\gamma_1))=\eta(\gamma_2)\rho(\gamma_1)\eta(\gamma_2)^{-1},\end{equation} which highlights the non-commutativity of the $\Gamma_1$ and $\Gamma_2$ actions if and only if $\mu$ is nontrivial.

In this work, we assume that $\Gamma_1$ and $\Gamma_2$ admit a semidirect product $\Gamma = \Gamma_1 \rtimes \Gamma_2$ with a representation $\varrho(\gamma_1,\gamma_2) = \rho(\gamma_1)\eta(\gamma_2)$ under the condition (\ref{action}). In this case
\begin{equation*} \label{relation}(\gamma_1, \gamma_2) v = \gamma_2 (\mu(\gamma_2^{-1})\gamma_1 v),
\end{equation*}
for all $(\gamma_1,\gamma_2) \in \Gamma$ and $v \in V.$ To simplify notation, from now on we shall write each representation $\rho(\gamma_1)$ and 
$\eta(\gamma_2)$  by $\gamma_1$ and $\gamma_2$, respectively. We also consider $\Gamma_1$ and $\Gamma_2$ endowed with epimorphisms
\begin{equation*}\label{defsigmas}
\sigma_1: \Gamma_1 \to \z_2 \quad \text{and} \quad \sigma_2: \Gamma_2 \to \z_2. \end{equation*}

  We now construct a mapping on $\Gamma= \Gamma_1 \rtimes \Gamma_2 $ in order to ``preserve''  what is a symmetry and what is a reversing symmetry on each component. This can be done in a natural way by the produt epimorphism,
\begin{equation}\label{sigma}
\begin{tabular}{cccl}
$\sigma:$ &$ \Gamma_1 \rtimes \Gamma_2 $ & $\rightarrow$ & $\z_2$ \\
          &$(\gamma_1, \gamma_2)$        &$\mapsto $     &$\sigma_1(\gamma_1) \sigma_2(\gamma_2)$.
\end{tabular}
\end{equation}
We notice that $\sigma$ is a group homomorphism if, and only if, for each $\gamma_2 \in \Gamma_2,$ the automorphism $\mu(\gamma_2)$ preserves the symmetries and reversing symmetries of $\Gamma_1,$ that is, $\sigma_1(\mu(\gamma_2) \gamma_1) = \sigma_1(\gamma_1),$ for all $(\gamma_1,\gamma_2) \in \Gamma$. In this work, this invariance of $\sigma_1$ is assumed throughout.\\

For the bireversible systems treated in Section~\ref{sec:NORMAL FORMS}, the groups $\Gamma_1$ and $\Gamma_2$ above are two commuting distinct  representations of ${\bf Z}_2$. In what follows we restrict to the case  $\Gamma_2 = \z_2^{\kappa}$, generated by a reversing symmetry $\kappa$, so that $\Gamma_1 \rtimes \z_2^{\kappa} / \Gamma_1$ is isomorphic to $\z_2 = \{\pm 1\}.$ In this case, we have assured the existence of the epimorphism $\tilde{\sigma} : \Gamma_1 \rtimes \z_2^{\kappa} \to \z_2,$ 
\begin{equation}\label{sigmatilde}
\tilde{\sigma}(\gamma_1, \gamma_2) = \sigma_2(\gamma_2).
\end{equation} 
This construction is a way to look at $\Gamma_1$ as a group whose elements act as symmetries,  since $\Gamma_1= \ker (\tilde{\sigma})$, and  it is an intermediate step  to obtain generators of $\P(\Gamma_1)$ as a module over $\P(\Gamma_1 \rtimes \z_2^{\kappa})$ (for the proof of Theorem~\ref{thm:SL} below). For that, we define the operators $\tilde{R}, \tilde{S}: \P(\Gamma_1) \to \P(\Gamma_1)$ by


\begin{equation}\label{R&S}
\tilde{R}(f)(v) = \frac{1}{2}\big(f(v) + f(\kappa v) \big) \quad \text{and} \quad \tilde{S}(f)(v) = \frac{1}{2}\big(f(v) - f(\kappa v) \big).
\end{equation} 

\noindent These are  homomorphisms of $\P(\Gamma_1 \rtimes \z_2^{\kappa})$-modules, and  $\tilde{R}$ corresponds to the Reynolds operator  defined in \cite{manoel2} used to produce a Hilbert basis for the ring of the invariants from a Hilbert basis for the invariants under a subgroup of index 2. Hence, by \cite[Theorem 3.2]{manoel2}, the set $\big \{\tilde{R}(u_i), \tilde{S}(u_i)\tilde{S}(u_j), 1 \leq i, j \leq s \big \}$ forms a Hilbert basis of the ring $\P(\Gamma_1 \rtimes \z_2^{\kappa}).$ Also, the operator $\tilde{S}$ corresponds to the Reynolds operator defined in \cite{manoel1} used in the construction of a generating set for the anti-invariant polynomial functions as a module over the ring of the invariants.  By \cite[Proposition 2.3]{manoel1}, $\tilde{S}$ is an idempotent projection and 
$$\P(\Gamma_1) = \ker \tilde{S} \oplus \mathrm 
{Im} (\tilde{S})$$
is a decomposition in  $\P(\Gamma_1 \rtimes \z_2^{\kappa})$-modules. We then have the following:



\begin{lemma} 
Let $\tilde{\sigma}$ defined as in (\ref{sigmatilde}). Then
\begin{equation}\label{decomposition}
 \P(\Gamma_1) = \P(\Gamma_1 \rtimes \z_2^{\kappa}) \oplus \Q_{\tilde{\sigma}}(\Gamma_1 \rtimes \z_2^{\kappa})
 \end{equation}
 
\noindent as modules over $\P(\Gamma_1 \rtimes \z_2^{\kappa}).$
\end{lemma}

\begin{proof}
From (\ref{R&S}) it is immediate that $\ker \tilde{S} = \P(\Gamma_1) \cap \P(\z_2^{\kappa}) = \P(\Gamma_1 \rtimes \z_2^{\kappa}).$ We now show that $\mathrm {Im}(\tilde{S}) = \Q_{\tilde\sigma}(\Gamma_1 \rtimes \z_2^{\kappa}).$ For that, we use \cite[Proposition 3.2]{BMZ} which states that $\Q_{\tilde\sigma}(\Gamma_1 \rtimes \z_2^{\kappa}) = \P(\Gamma_1) \cap \Q_{\sigma_2}(\z_2^{\kappa}).$ As $\mathrm {Im}( \tilde{S}) \subseteq \P(\Gamma_1)$ and
$$ \tilde{S}(f)(\kappa v) = \frac{1}{2}\big(f(\kappa v) - f(\kappa^2 v) \big) = \frac{1}{2}\big(f(\kappa v) - f(v) \big) \nonumber = - \tilde{S}(f)(v),$$ it follows that $\mathrm {Im}(\tilde{S}) \subseteq \Q_{\tilde\sigma}(\Gamma_1 \rtimes \z_2^{\kappa}).$ Now, if $f \in \Q_{\tilde\sigma}(\Gamma_1 \rtimes \z_2^{\kappa}),$
then $$ \tilde{S}(f)(v) = \frac{1}{2}\big(f(v) - f(\kappa v) \big) = \frac{1}{2}\big(f(v) + f(v) \big) = f(v),$$ that is, $\mathrm {Im}(\tilde{S}) = \Q_{\tilde\sigma}(\Gamma_1 \rtimes \z_2^{\kappa}).$ 
\end{proof}




\quad

From \cite[Corollary 3.3]{manoel1} and decomposition (\ref{decomposition}), if $\{u_1, \ldots, u_s\}$ is a Hilbert basis for $\P(\Gamma_1),$ then $$\{\tilde{S}(u_0) \equiv 1, \tilde{S}(u_1), \ldots, \tilde{S}(u_s)\}$$ is a set of generators of $\P(\Gamma_1)$ as a module over $\P(\Gamma_1 \rtimes \z_2^{\kappa}).$ Consider now $\QQ_{\sigma_1}(\Gamma_1)$ the module of $\Gamma_1-$reversible-equivariants under $\P(\Gamma_1).$ As $\Gamma_2 = \z_2^{\kappa}$ we have a finite number of generators of $\QQ_{\sigma_1}(\Gamma_1)$ over $\P(\Gamma_1  \rtimes \z_2^{\kappa}),$  as shown below:\\

\begin{theorem}
\label{thm:SL}
 Let $\{u_1, \ldots, u_s\}$ be a Hilbert basis for the ring $\P(\Gamma_1)$ and let $\{L_0, \ldots, L_r\}$ be generators of $\QQ_{\sigma_1}(\Gamma_1)$ over $\P(\Gamma_1).$ 
 Then $$\{\tilde{S}(u_i)L_j, \ 0 \leq i \le s, \ 0 \leq j \le r\}$$ generates $\QQ_{\sigma_1}(\Gamma_1)$ over $\P(\Gamma_1 \rtimes \z_2^{\kappa}).$
\end{theorem}

\begin{proof} As we just mentioned above, $\{\tilde{S}(u_0) \equiv 1, \tilde{S}(u_1), \ldots, \tilde{S}(u_s)\}$ is a set of generators for the module $\P(\Gamma_1)$ over $\P(\Gamma_1 \rtimes \z_2^{\kappa}).$ The proof now follows exactly the same steps of \cite[Lemma 3.4]{manoel1}.
\end{proof}

\quad 

We now define the operator $T : \QQ_{\sigma_1}(\Gamma_1)\to \QQ_{\sigma_1}(\Gamma_1)$ by

\begin{equation} \label{eq:T}
T(G) = \frac{1}{2}\left(G - \kappa G \kappa \right).
\end{equation}
We then have:






\begin{lemma} \label{lemma}
The mapping $T$ is an homomorphism of modules over the ring $\P(\Gamma_1 \rtimes \z_2^{\kappa}).$ Moreover, $T$ is an idempotent projection with $\mathrm{Im}(T) = \QQ_\sigma(\Gamma_1 \rtimes \z_2^{\kappa}).$
\end{lemma}

\begin{proof}
To prove that $T$ is an homomorphism of $\P(\Gamma_1 \rtimes \z_2^{\kappa})$-modules we use that $\z_2^{\kappa}$-action is linear and that $\P(\Gamma_1 \rtimes \z_2^{\kappa}) = \P(\Gamma_1) \cap \P(\z_2^{\kappa})$ (see \cite[Proposition 3.2]{BMZ}). To prove that $\mathrm{Im}(T) = \QQ_\sigma(\Gamma_1 \rtimes \z_2^{\kappa})$ we first prove  that $T(G) \in \QQ_\sigma(\Gamma_1 \rtimes \z_2^{\kappa}),$ for all $G \in \QQ_{\sigma_1}(\Gamma_1).$ But, again by \cite[Proposition 3.2]{BMZ}, $\QQ_\sigma(\Gamma_1 \rtimes \z_2^{\kappa}) = \QQ_{\sigma_1}(\Gamma_1) \cap \QQ_{\sigma_2}(\z_2^{\kappa}),$ so it suffices to show that $T(G)(\kappa v) = - \kappa T(G)(v),$ for all $G \in \QQ_{\sigma_1}(\Gamma_1)$ and $v \in V.$ In fact, 
$$\kappa T(G)(\kappa v) = \kappa \biggr (\frac{1}{2} \bigr(G (\kappa v) - \kappa G (\kappa^2 v) \bigl) \biggl ) = \frac{1}{2} \left(\kappa G (\kappa v) - G (v) \right) = - T(G)(v). 
$$
Therefore, $T(G) \in \QQ_\sigma(\Gamma_1 \rtimes \z_2^{\kappa}).$ Now, let $G \in \QQ_\sigma(\Gamma_1 \rtimes \z_2^{\kappa}) = \QQ_{\sigma_1}(\Gamma_1) \cap \QQ_{\sigma_2}(\z_2^{\kappa}).$ Then  
$$ T(G)(v) = \frac{1}{2} \bigr(G (v) - \kappa G (\kappa v) \bigl) = \frac{1}{2} \bigr(G (v) + \kappa^2 G (v) \bigl) = G(v).$$
Thus, $\mathrm{Im}(T) = \QQ_\sigma(\Gamma_1 \rtimes \z_2^{\kappa})$ and the restriction of $T$ to $\mathrm{Im}(T)$ is the identity,  implying that $T$ is an idempotent projection.
\end{proof}

\quad 

The following result is now a direct consequence of the last proposition.

\begin{theorem}
\label{thm:GENERATORS_BY_T}
If $\QQ_{\sigma_1}(\Gamma_1)$ is a finitely generated module over $\P(\Gamma_1 \rtimes \z_2^{\kappa})$ with generators $G_1, \ldots , G_n,$ then $\{T(G_i): 1 \leq i \leq n\}$ generates $\QQ_\sigma(\Gamma_1 \rtimes \z_2^{\kappa})$ over $\P(\Gamma_1 \rtimes \z_2^{\kappa}).$
\end{theorem}

\begin{proof}
By Lemma~\ref{lemma}, the direct sum decomposition
$$\QQ_{\sigma_1}(\Gamma_1) = \ker T \oplus \QQ_\sigma(\Gamma_1 \rtimes \z_2^{\kappa})$$ of modules over the ring $\P(\Gamma_1 \rtimes \z_2^{\kappa})$ holds. Given $\tilde{G} \in \QQ_\sigma(\Gamma_1 \rtimes \z_2^{\kappa}) = \mathrm{Im}(T),$ there exists $G \in \QQ_{\sigma_1}(\Gamma_1)$ such that $T(G) = \tilde{G}.$ Write $G = \sum_{i = 1}^{n} f_i G_i,$ where $f_i \in \P(\Gamma_1 \rtimes \z_2^{\kappa}).$ Hence,
\[\tilde{G} = T(G) = T\big(\sum_{i = 1}^{n} f_i G_i\big) = \sum_{i = 1}^{n} f_i T \big(G_i\big).\]
\end{proof}

\quad 

We remark that if we are to use Theorem~\ref{thm:GENERATORS_BY_T}  to find generators of $\QQ_\sigma(\Gamma_1 \rtimes \z_2^{\kappa})$ over $\P(\Gamma_1 \rtimes \z_2^{\kappa})$, we need  $\QQ_{\sigma_1}(\Gamma_1)$ to be a finitely generated module over the ring $\P(\Gamma_1 \rtimes \z_2^{\kappa})$. If this holds, then we just project its generators by the operator $T.$ \\

We end this section
presenting the computation of generators of reversible equivariants under a group of type $\Gamma_1 \rtimes \z_2^{\kappa}$ in an algorithmic way: \\

Input: \begin{itemize}
  \item Compact Lie group $\Gamma = \Gamma_1 \rtimes \z_2^{\kappa}$ and epimorphisms $\sigma_1: \Gamma_1 \to \z_2 $ and $\sigma_2:   \z_2^{\kappa}\to \z_2$,  
  where $\kappa$ is a reversing symmetry;
  \item Hilbert basis $\{u_1, \ldots, u_s\}$ for $\P(\Gamma_1);$
  \item Generating set $\{L_0, \ldots, L_r\}$ for $\QQ_{\sigma_1}(\Gamma_1)$ over $\P(\Gamma_1);$
\end{itemize}

\quad 

Output: Hilbert basis for $\P(\Gamma)$ and generating set for $\QQ_\sigma(\Gamma)$ over $\P(\Gamma)$ with $\sigma$ defined in \eqref{sigma}.

\quad

Procedure:

\begin{itemize}
\item[1.] Do $\tilde{S}(u_0) = 1;$
\item[2.] For $1 \leq i \leq s$ compute $\tilde{R}(u_i)$ and $\tilde{S}(u_i),$ where $\tilde{R},$ $\tilde{S}$ are given by (\ref{R&S});
\item[3.] The set $\big \{\tilde{R}(u_i), \tilde{S}(u_i)\tilde{S}(u_j), 1 \leq i, j \leq s \big \}$ forms a Hilbert basis of the ring $\P(\Gamma_1 \rtimes \z_2^{\kappa})$ (Using \cite[Theorem 3.2]{manoel2});
\item[4.] The set $\{\tilde{S}(u_i)L_j, \ 0 \leq i \le s, \ 0 \leq j \le r\}$ generates $\QQ_{\sigma_1}(\Gamma_1)$ over $\P(\Gamma_1 \rtimes \z_2^{\kappa})$ (Using Theorem~\ref{thm:SL});
\item[5.] The set $\{T(\tilde{S}(u_i)L_j), \ 0 \leq i \le s, \ 0 \leq j \le r\}$ generates $\QQ_\sigma(\Gamma_1 \rtimes \z_2^{\kappa})$ over $\P(\Gamma_1 \rtimes \z_2^{\kappa}),$ where $T$ is given by (\ref{eq:T}) (Using Theorem \ref{thm:GENERATORS_BY_T}).

\end{itemize}

\section{Computing normal forms} \label{sec:NORMAL FORMS}

In this section we apply the algorithm of the previous section to deduce the normal forms of for types of bireversible vector fields defined on $\R^{2n+2}$. We consider a special type of linearization, under  both  resonance and non-resonance conditions. We consider $\Gamma-$reversible-equivariant systems
\begin{equation}\label{eq:SYSTEM}
  \dot{x} = X(x)
\end{equation}
whose linearization about the origin has matrix $L$ of type (\ref{L})  with nonzero $\omega_i,$ $i=1, \cdots, n.$ The vector fields are reversible equivariant under the group $\z_2 \times \z_2,$  generated by two linear commuting involutions $\phi$ and $\psi$. 

\subsection{Characterization of the involutions} \label{subset:characterization}

In this subsection, we characterize the possible pairs of linear involutions that generate  $\z_2 \times \z_2$ up to the equivalence given by  simultaneous conjugacy: 
 two pairs of linear involutions $(\phi, \psi)$ and $({\tilde{\phi}}, \tilde{\psi})$ on $\R^n$ are said to be {\it equivalent} if there exists a linear diffeomorphism $H$ on $\R^n$ such that 
\begin{equation} \label{eq:equivalence}
\tilde{\phi}= H \circ \phi \circ H^{-1}, \ \ \tilde{\psi}= H \circ \psi\circ H^{-1}.
\end{equation} 
The matricial normal forms of pairs  of involutions $(\phi, \psi)$ up to this equivalence must  anti-commute with  $L$ given by (\ref{L}), that is, 
\[L \phi = - \phi L, \quad \quad L \psi = - \psi L. \] 
Direct computations show that
the matrix of any linear involution $\varphi$ that anti-commutes with $L$  is block diagonal of the form
 \begin{equation}
\label{involution}
\varphi= \mathsf{diag} (\varphi_0, \ldots, \varphi_n),
\end{equation}
where
\begin{equation*}
\varphi_0=\begin{pmatrix}  a_0 & 0 \\  0  & -a_0 \end{pmatrix}, \quad  \quad \varphi_i=\begin{pmatrix}  a_i & b_i \\  b_i  & -a_i \end{pmatrix},
\end{equation*}
with $a_0^2=1$ and $a_i^2 + b_i^2=1$ for $1 \leq i \leq n$. Hence, each $\varphi_i$ is a reflection of order $2$. Now, if we denote by $\fix(\varphi)$ the {\it fixed-point space} for $\varphi$,
\[ \fix(\varphi)= \left\{ x\in \R^n,~ \varphi(x)=x \right\}, \]
we have $\dim \fix(\varphi_i)=1,$ $0 \leq i \leq n.$ Therefore,  $\dim \fix (\varphi)= n+1.$

Let $(\phi, \psi)$ and $(\tilde{\phi}, \tilde{\psi})$ be two pairs of linear involutions generating an Abelian group. The matricial form of each is of the form  \eqref{involution}. If they are equivalent, then the linear isomorphism $H$ of (\ref{eq:equivalence}) must commute with $L$, since  the linear part of the system must be preserved under equivalence. Therefore $H$ has  block diagonal matrix \begin{equation*}
\mathsf{diag} (H_0, \ldots, H_n),
\end{equation*}
where $H_i$ are invertible matrices of order $2$ giving the equivalence between $(\phi_i, \psi_i)$ and $(\tilde{\phi_i}, \tilde{\psi_i}),$ for $ 0 \leq i \leq n$. It follows that the classification of pairs of commuting involutions on $\R^{2n+2}$  that anti-commute with $L$ is reduced to the classification of pairs of reflections on $\R^2$ whose fixed-point subspaces coincide or are orthogonal straight lines. 
In the first case, $\fix(\phi_i)$ and $\fix (\psi_i)$ coincide and, therefore, $\phi_i=\psi_i$ for $0 \leq i \leq n$. In the second case these are in particular  transversal lines, so from \cite[Teorema 6.2]{mancini} it follows that each pair $(\phi_i,\psi_i)$, $ 0 \leq i \leq n,$ is equivalent to the pair
\begin{equation*}
\begin{pmatrix} 1 & 0 \\ 0 &-1 \end{pmatrix}, \quad \begin{pmatrix} -1 & 0 \\ 0 &1 \end{pmatrix}.
\end{equation*} 
Therefore, up to equivalence,  there are $2^{n}$ pairs of  involutions $(\phi,\psi)$ 
that can generate non-conjugate copies of $\z_2\times \z_2$, namely those with diagonal matricial form 
\begin{equation*} 
\label{phi-psi}
\phi= \mathsf{diag}(1,-1,\ldots,1,-1)\quad \quad 
\psi= \mathsf{diag}(a_0, -a_0, \ldots, a_n, -a_n),
\end{equation*}
\noindent with $a_i= \pm 1, $  $0 \leq i \leq n$. 

From now on we use complex coordinates of ${\R}^{2n} \simeq {\C}^n$ with the action of $\z_2\times \z_2$ on $\R^2 \times \C^n$ given by
\begin{equation}
\label{phiaction}
\phi(x_1, x_2, z_1, \ldots, z_n)= (x_1, -x_2, \bar{z}_1, \ldots, \bar{z}_n),
\end{equation}
\begin{equation}
\label{psiactiononC}
\psi(x_1, x_2, z_1, \ldots, z_n)= (a_0 x_1, - a_0 x_2, a_1 \bar{z}_1, \ldots, a_n \bar{z}_n).
\end{equation}

\noindent Notice that when $a_i=1,~0 \leq i \leq n$ we have $\phi=\psi$. 

\subsection{Non resonant case}
In this subsection, we compute the normal form of bireversible systems with linear part $L$ given in (\ref{L}) where $\omega_1, \dots, \omega_n$ are algebraically independent. 

It follows from Proposition~\ref{prop:SFORM} that in this case we have $\s = \r \times \T^n,$ where the action of $\s$ is given by (see \cite{BMZ}) 
\begin{equation}
\label{S-actiononR2}
\mathbf{s}(x_1,x_2)= (x_1, sx_1 + x_2)
\end{equation} and \begin{equation*}
\theta (z_1,\ldots, z_n)= (e^{i\theta_1}z_1, \ldots, e^{i\theta_n}z_n),
\end{equation*}
for $s\in \r$ and $ \theta = (\theta_1, \ldots, \theta_n) \in \T^n$.

It is straightforward  that $\PP(\r)$ is generated by
\begin{equation}
\label{basisR-R2}
\left\{ (x_1, x_2), ~( 0 , 1)  \right\},
\end{equation}
over the ring $\P(\r)= \left< x_1 \right>$. By \cite[XII Example 4.1(b)]{golub} and  \cite[Lemma 3.3]{BMZ}, a Hilbert basis for $\P(\s)$ is given by $\{v_1, \dots, v_{n+1}\},$ where
\[v_1(x,z) = x_1, \quad v_2(x,z) = |z_1|^2, \quad \dots, \quad v_{n+1}(x,z) = |z_n|^2 ,\]

\noindent with $x = (x_1,x_2) \in \mathbb{R}^2$ and $z = (z_1, \dots, z_n).$ Again, by \cite[XII Example 5.4(a)]{golub}, (\ref{basisR-R2}) and \cite[Lemma 3.3]{BMZ}, $\PP(\s)$ is generated over the ring $\P(\s)$ by the mappings:\\
$H_0(x,z) = (x_1, x_2, 0,  \dots, 0)$, \quad $H_1(x,z) = (0,1,0, \dots, 0)$, \\
$H_2(x,z) = (0,0, z_1,\dots, 0)$, \quad $H_3(x,z) = (0,0, iz_1,\dots, 0), \cdots$,  \\
$H_{2n}(x,z) = (0, \dots, 0, z_n)$ and $H_{2n+1}(x,z) = (0, \dots, 0, iz_n)$.

Let us now denote by $\z_2^\phi$ and $\z_2^\psi$ each copy of $\z_2$ generated by $\phi$ and $\psi$ respectively. Consider the epimorphisms $\sigma_1: \s \rtimes \z_2^\phi \to \z_2$ and $\sigma_2: \z_2^\psi \to \z_2$ defined respectively as 
\begin{equation}
\label{eq:EPI}
\sigma_1(s,\phi) = - 1 \quad \text{and} \quad \sigma_2(\psi)=-1,
\end{equation}

\noindent for all $s \in \s,$ and we consider $\sigma: (\s \rtimes \z_2^\phi) \rtimes \z_2^\psi \to \z_2$ as defined in \eqref{sigma}. By Theorem \ref{thm:NORMALFORM} we need to find generators of $\QQ_\sigma(\s \rtimes (\z_2^\phi \times \z_2^\psi))$. For that,  we use the algorithm given in Section~\ref{sec:SEMIDIRECT}:

\begin{itemize}
  \item[1.] We start by computing a Hilbert basis for $\P(\s \rtimes \z_2^{\phi}):$ we consider the operators $\tilde{R}, \tilde{S}: \P(\s) \to \P(\s)$ as in (\ref{R&S}) for $\kappa = \phi$. We have that $v_i \in \P(\z_2^{\phi}),$  $1 \leq i \leq n+1.$ Thus $\tilde{R}(v_i) = v_i$ and $\tilde{S}(v_i) = 0,$  $1 \leq i \leq n+1.$ By \cite[Theorem 3.2]{manoel2}, we have
\begin{equation}\label{basis1}
  \P(\s \rtimes \z_2^{\phi}) = \P(\s) = \langle v_1, \dots, v_{n+1} \rangle.
\end{equation}

  \item[2.] Consider $\check{S} : \P(\s \rtimes \z_2^\phi) \to \P(\s \rtimes \z_2^\phi)$ as in (\ref{R&S}) with $\kappa = \psi$ given by (\ref{psiactiononC}). For $i \geq 2,$ each $v_i$ is invariant under $\psi.$ Therefore, $\check{S}(v_i) = 0,$ for all $i \geq 2.$ Moreover, $\check{S}(v_1) = (1-a_0)v_1/2.$ 
  


\item[3.] To obtain the generators of $\QQ_{\sigma_1}(\s \rtimes \z_2^\phi)$ over the ring $\P(\s \rtimes \z_2^\phi),$ we first  find generators of $\QQ_{\sigma_1}(\s \times Id)$ over $\P(\s \rtimes \z_2^\phi)$ and then project them by $T : \QQ_{\sigma_1}(\s \times Id) \to \QQ_{\sigma_1}(\s \times Id)$ defined in (\ref{eq:T}), with $\kappa = \phi$ (Theorem \ref{thm:GENERATORS_BY_T}).
  
The group $\s$ has only symmetries, so $\QQ_{\sigma_1}(\s \times Id) = \PP(\s)$ and since $\P(\s) = \P(\s \rtimes \z_2^\phi)$ the generators of $\QQ_{\sigma_1}(\s \times Id)$ over $\P(\s \rtimes \z_2^\phi)$ are $H_0, \dots, H_{2n+1}.$ For $j$ even, $H_j$ is equivariant under $\phi$ and, in this case, $T(H_j) = 0.$ For $j$ odd, $H_j$ is reversible under $\phi$ and, in this case, $T(H_j) = H_j.$ By Theorem~\ref{thm:GENERATORS_BY_T}, 
\[L_0 \equiv H_1, L_1 \equiv H_3, \dots, L_n \equiv H_{2n+1} \] generate $\QQ_{\sigma_1}(\s \rtimes \z_2^\phi)$ over the ring $\P(\s \rtimes \z_2^\phi).$


\item[4.] Define $\check{T} : \QQ_{\sigma_1}(\s \rtimes \z_2^\phi) \to \QQ_{\sigma_1}(\s \rtimes \z_2^\phi)$  as  in (\ref{eq:T}) with $\kappa = \psi$ given by (\ref{psiactiononC}). Again by Theorem~\ref{thm:GENERATORS_BY_T}, to obtain the generators of $\QQ_\sigma(\s \rtimes (\z_2^\phi \times \z_2^\psi))$ over the ring $\P(\s \rtimes (\z_2^\phi \times \z_2^\psi)),$ we first find the generators of $\QQ_{\sigma_1}(\s \rtimes \z_2^\phi)$ over $\P(\s \rtimes (\z_2^\phi \times \z_2^\psi))$ and then project by $\check{T}.$ 

    
Consider $\check{S}$ as in the item 2 above and define $\check{S}(v_0) \equiv 1.$ By Theorem~\ref{thm:SL}, the set $\{\check{S}(v_i)L_j, \ i = 0, 1, \ 0 \leq j \le n\}$ generates the module $\QQ_{\sigma_1}(\s \rtimes \z_2^\phi)$ over $\P(\s \rtimes (\z_2^\phi \times \z_2^\psi)).$ Thus, we compute $\check{L}_{0j} = \check{T}(\check{S}(v_0)L_j) = \check{T}(L_j)$ and $\check{L}_{1j} = \check{T}(\check{S}(v_1)L_j)$ for $0 \leq j \leq n.$ It is direct that  $\check{L}_{0j} = L_j$ for $j \geq 1.$ Moreover, eliminating constants, we have
\[\check{L}_{00} = (1 + a_0)L_0, \qquad \check{L}_{10} = (1 - a_0)v_1L_0\] and $\check{L}_{1j} \equiv 0$ for $j \geq 1.$
  
When $a_0 = 1$ the set $\{L_0, \ldots, L_n\}$ generates $\QQ_\sigma(\s \rtimes (\z_2^\phi \times \z_2^\psi))$ over the ring $\P(\s \rtimes (\z_2^\phi \times \z_2^\psi))$ and when $a_0 = -1$ the set $\{v_1L_0, L_1, \ldots, L_n\}$ generates $\QQ_\sigma(\s \rtimes (\z_2^\phi \times \z_2^\psi))$ over $\P(\s \rtimes (\z_2^\phi \times \z_2^\psi)).$
\end{itemize}

It remains to find a Hilbert basis for $\P(\s \rtimes (\z_2^\phi \times \z_2^\psi)).$ For this we take the
Hilbert basis $\{v_1, \ldots v_{n+1}\}$ for $\P(\s \rtimes \z_2^\phi)$ given in (\ref{basis1}) and apply \cite[Theorem 3.2]{manoel2}, by considering the operators $\check{R}, \check{S} : \P(\s \rtimes \z_2^\phi) \to \P(\s \rtimes \z_2^\phi)$ as defined in (\ref{R&S}) with $\kappa = \psi$ ($\check{S}$ is the operator defined in  item 2 above). Then, eliminating constants, we have
\[\check{R}(v_1) = (1 + a_0)v_1 \qquad \text{e} \qquad \check{S}(v_1) = (1-a_0)v_1.\] Moreover, $\check{R}(v_i) = v_i$ and $\check{S}(v_i) = 0$ for $i \geq 2.$ Thus, for $a_0 = 1$ we have

$$\P(\s \rtimes (\z_2^\phi \times \z_2^\psi)) = \langle v_1, \ldots , v_{n+1} \rangle,$$ and for $a_0 = -1$
$$\P(\s \rtimes (\z_2^\phi \times \z_2^\psi)) = \langle v_1^2, v_2, \ldots , v_{n+1} \rangle .$$

Therefore, if $a_0=1$ then $\phi=\psi$ and, in this case, we obtain the $\z_2^\phi-$reversible normal form
\begin{small} 
\begin{align}
\label{NF-Z2}
\dot{x}_1 &= x_2 \nonumber \\
\dot{x}_2 &= f_0(x_1, |z_1|^2, \ldots, |z_n|^n )\nonumber\\
\dot{z}_1 &= -i\omega_1z_1 + iz_1f_1(x_1, |z_1|^2, \ldots, |z_n|^n )\\
\dot{z}_2 &= -i\omega_2z_2 + iz_2f_2(x_1, |z_1|^2, \ldots, |z_n|^n )\nonumber \\
\vdots &  \nonumber \\
\dot{z}_n &= -i\omega_2z_n + iz_nf_n(x_1, |z_1|^2, \ldots, |z_n|^n ), \nonumber
\end{align}
\end{small}
for  $f_i: \R^{n+1},0 \to \R$, $0 \leq i \leq n$. If $a_0=-1$, then the $\z_2^\phi\times\z_2^\psi-$reversible-equivariant normal form is
\begin{small}
\begin{align*}
\dot{x}_1 &= x_2 \nonumber\\
\dot{x}_2 &= x_1f_0(x_1^2, |z_1|^2, \ldots, |z_n|^n )\nonumber\\
\dot{z}_1 &= -i\omega_1z_1 + iz_1f_1(x_1^2, |z_1|^2, \ldots, |z_n|^n )\\
\dot{z}_2 &= -i\omega_2z_2 + iz_2f_2(x_1^2, |z_1|^2, \ldots, |z_n|^n )\nonumber \\
\vdots &  \nonumber \\
\dot{z}_n &= -i\omega_2z_n + iz_nf_n(x_1^2, |z_1|^2, \ldots, |z_n|^n ),\nonumber
\end{align*}
\end{small}
for  $f_i: \R^{n+1},0 \to \R$, $0 \leq i \leq n$.

\begin{remark} {\rm The normal forms obtained above depend only on $a_0.$ This is due to the fact that each $v_i,$ $2 \leq i \leq n+1,$ is $\psi-$invariant in the $z-$coordinate. We also remark that the normal form \eqref{NF-Z2} has been obtained by Lima and Teixeira in \cite{lima}. The authors use the classical method developed by Belitskii in the context without symmetries to get a pre-normal form and impose the reversibility of $\phi$ {\it a posteriori}. In our approach, the procedure takes the reversibility into consideration from the beginning, which simplifies the process of annihilating terms up to equivalence.}
\end{remark}

\subsection{Resonance of type $(n_1: n_2: 0)$ in $\R^2 \times \C^3$} 

The actions of $\z_2^\phi$ and $\z_2^\psi$ on $\R^2 \times \C^3$ is given by \eqref{phiaction} and \eqref{psiactiononC} for $n=3$. We assume that $n_1 \omega_2 - n_2 \omega_1 = 0,$ with $n_1, n_2 \in \mathbb{N}$ nonzero, and $\omega_3$ is algebraically independent with respect to $\omega_1$ and $\omega_2.$ In this case, the system (\ref{eq:SYSTEM}) is called $(n_1:n_2:0)-$resonant. 

By Proposition~\ref{prop:SFORM}, $\s = \r \times \T^2.$ The diagonal action of $\s$ on $\R^2 \times \C^3$ is given from the standard action of $\r$ on $\R^2$ as in (\ref{S-actiononR2}) and from the action of $\T^2$ on $\C^3$ given by
\begin{equation*}
(\theta_1, \theta_2) (z_1,z_2,z_3)= (e^{in_1\theta_1}z_1, e^{in_2\theta_1}z_2, e^{i\theta_2}z_3), \quad (\theta_1, \theta_2) \in \T^2.
\end{equation*}

By \cite[XII Example 5.4(a)]{golub}, \cite[XIX Theeorem 4.2)]{golub} and  \cite[Lemma 3.3]{BMZ}, a Hilbert basis for $\P(\s)$ is given by $\{v_1, \dots, v_6\},$ where
\[v_1(x,z) = x_1, \ \  v_2(x,z) = |z_1|^2, \ \ v_3(x,z) = |z_2|^2, \ \ v_4(x,z)= \re(z_1^{n_2}\bar{z}_2^{n_1}),\] \[v_5(x,z)=\im(z_1^{n_2}\bar{z}_2^{n_1}), \ \ v_6(x,z) = |z_3|^2,\] 

\noindent with $x = (x_1,x_2) \in \mathbb{R}^2$ and $z = (z_1, z_2, z_3) \in \C^3.$ Moreover, $\PP(\s)$ is gene\-rated over $\P(\s)$ by the mappings \[H_0(x,z) = (x_1, x_2, 0, 0, 0), \ \ H_1(x,z) = (0,1,0, 0, 0),\] \[H_2(x,z) = (0,0, z_1, 0, 0), \ \ H_3(x,z) = (0,0, iz_1, 0, 0), \] \[ H_4(x,z) = (0, 0, \bar{z_1}^{n_2 - 1}z_2^{n_1}, 0,0), \ \ H_5(x,z) = (0, 0, i\bar{z_1}^{n_2 - 1}z_2^{n_1}, 0,0),\]
\[H_6(x,z) = (0,0, 0, z_2,0), \ \ H_7(x,z) = (0,0, 0, iz_2, 0), \] \[ H_8(x,z) = (0, 0, 0, z_1^{n_2}\bar{z}_2^{n_1 - 1}, 0), \ \ H_9(x,z) = (0, 0, 0, i z_1^{n_2}\bar{z}_2^{n_1 - 1}, 0),\]
\[H_{10}(x,z) = (0,0, 0, 0, z_3), \ \ H_{11}(x,z) = (0,0, 0, 0, iz_3).\]

Next we consider the epimorphisms $\sigma_1$ and $\sigma_2$ given in (\ref{eq:EPI}), and $\sigma$ as defined in \eqref{sigma}. We determine now a set of generators for $\QQ_\sigma(\s \rtimes (\z_2^\phi \times \z_2^\psi))$ over $\P(\s \rtimes (\z_2^\phi \times \z_2^\psi))$ as follows:

\begin{itemize}
  \item[1.] We start getting a Hilbert basis for $\P(\s \rtimes \z_2^{\phi}).$ Again we consider the operators $\tilde{R}, \tilde{S}: \P(\s) \to \P(\s)$ defined in (\ref{R&S}) with $\kappa = \phi$. 
We have $\tilde{R}(v_i) = v_i$ and $\tilde{S}(v_i) = 0,$ for all $i \neq 5,$ and $\tilde{R}(v_5) = 0,$ $\tilde{S}(v_5) = v_5.$ By \cite[Theorem 3.2]{manoel2}, $\P(\s \rtimes \z_2^{\phi}) = \langle v_1, \dots, v_4, v_5^2, v_6 \rangle.$ Since $v_5^2 = v_2^{n_2}v_3^{n_1} - v_4^2$, we have \begin{equation*}
  \P(\s \rtimes \z_2^{\phi}) = \langle u_1, \dots, u_5 \rangle,
\end{equation*}
where $u_i = v_i$ for $i = 1, \ldots, 4$ and $u_5 = v_6.$

\item[2.] Define $\check{S} : \P(\s \rtimes \z_2^\phi) \to \P(\s \rtimes \z_2^\phi)$ as in (\ref{R&S}) with $\kappa = \psi$ given in (\ref{psiactiononC}). We have  $\check{S}(u_i) = 0,$ for $i = 2, 3, 5,$ $\check{S}(u_1) = (1-a_0) u_1$ and $\check{S}(u_4) = (1-a_1^{n_2}a_2^{n_1})u_4.$

\item[3.] Here we determine the generators for $\QQ_{\sigma_1}(\s \rtimes \z_2^\phi)$ over the ring $\P(\s \rtimes \z_2^\phi).$ As in the last subsection, we first obtain the generators of $\PP(\s)$ over $\P(\s \rtimes \z_2^\phi).$ 

Since $\s$ has only symmetries, then $\QQ_{\sigma_1}(\s \times Id) = \PP(\s).$  Thus $\{H_0, \ldots, H_{11}\}$ generates $\QQ_{\sigma_1}(\s \times Id)$ over $\P(\s) = \langle v_1, \ldots, v_6 \rangle.$ We now consider $\tilde{S}$ defined in step 1 above. From \cite[Corollary 3.3]{manoel1}, $\bigl \{\tilde{S}(v_0) \equiv 1, \tilde{S}(v_5) \equiv v_5 \bigr\}$ is a set of generators for the module $\P(\s)$ over $\P(\s \rtimes \z_2^\phi).$ So 
$$\bigl \{H_j, \ v_5H_j: \ 0 \leq j \leq 11 \bigr \}$$ is a set of generators for $\QQ_{\sigma_1}(\s \times Id)$ over $\P(\s \rtimes \z_2^\phi).$ 

By Theorem~\ref{thm:GENERATORS_BY_T}, it  remains to project such generators by the mapping $T : \QQ_{\sigma_1}(\s \times Id)\to \QQ_{\sigma_1}(\s \times Id)$ defined in (\ref{eq:T}) for $\kappa = \phi$ acting as in (\ref{phiaction}). For $j$ even, $H_j$ is equivariant under $\phi$ and $v_5H_j$ is reversible under $\phi.$ In this case, $T(H_j) \equiv 0$ and $T(v_5H_j) \equiv v_5H_j.$ For $j$ odd, $H_j$ is reversible  under $\phi$ and $v_5H_j$ is equivariant under $\phi.$ In this case, $T(H_j) \equiv H_j$ and $T(v_5H_j) \equiv 0.$ Therefore, the elements 
\[L_0 \equiv H_1, \ L_1 \equiv v_5H_0, \ L_2 \equiv H_3, \ L_3 \equiv H_5, \ L_4 \equiv v_5H_2, \ L_5 \equiv v_5H_4,\] \[L_6 \equiv H_7, \ L_7 \equiv H_9, \ L_8 \equiv v_5H_6, \ L_9 \equiv v_5H_8, \ L_{10} \equiv H_{11}, \ L_{11} \equiv v_5H_{10}\] generate $\QQ_{\sigma_1}(\s \rtimes \z_2^\phi)$ over the ring $\P(\s \rtimes \z_2^\phi).$


 \item[4.] Consider now $\check{S}$ as in step 2 above and define $\check{T}: \QQ_{\sigma_1}(\s \rtimes \z_2^\phi) \to \QQ_{\sigma_1}(\s \rtimes \z_2^\phi)$ whose law is the same as of $T$ in (\ref{eq:T}) with $\kappa = \psi$ acting as in (\ref{psiactiononC}). Define $\check{S}(u_0) \equiv 1.$ By Theorem~\ref{thm:GENERATORS_BY_T}, the generators of $\QQ_\sigma(\s \rtimes (\z_2^\phi \times \z_2^\psi))$ over $\P(\s \rtimes (\z_2^\phi \times \z_2^\psi))$ are given by $\check{L}_{ij} = \check{T}(\check{S}(u_i)L_j)$ for $i=0,1,4$ and $0 \leq j \leq 11.$ 
 
Eliminating constants after projection, direct computations give
\[\check{L}_{00} \equiv (1 + a_0)L_0, \quad \check{L}_{0l} \equiv (1 + a_1^{n_2}a_2^{n_1})L_l, \quad \check{L}_{0k} \equiv L_k,\] \[\check{L}_{10} \equiv (1 - a_0)u_1L_0, \quad \check{L}_{1l} \equiv (1 + a_0a_1^{n_2}a_2^{n_1})(1 - a_0)u_1L_l, \quad \check{L}_{1k} \equiv 0,\] \[\check{L}_{40} \equiv (1 - a_0)(1 - a_1^{n_2}a_2^{n_1})u_4L_0, \quad \check{L}_{4l} \equiv (1 - a_1^{n_2}a_2^{n_1})u_4L_l, \quad \check{L}_{4k} \equiv 0,\] 

\noindent for $l = 1,3,4,7,8,11$ e $k = 2,5,6,9,10.$

\end{itemize}


\quad

We have four cases to consider (see Table \ref{table-R2C3}). In each case, we present a Hilbert basis for $\P(\s \rtimes (\z_2^\phi \times \z_2^\psi))$ by using \cite[Theorem 3.2]{manoel2} with $R, S: \P(\s \rtimes \z_2^\phi) \to \P(\s \rtimes \z_2^\phi)$ as in (\ref{R&S}) and $\kappa = \psi.$

\begin{itemize}
\item[Type A:] When $a_0 = 1$ and $a_1^{n_2}a_2^{n_1} = 1,$ a set of generators for $\QQ_\sigma(\s \rtimes (\z_2^\phi \times \z_2^\psi))$ over $\P(\s \rtimes (\z_2^\phi \times \z_2^\psi))$ is 
$$\{L_0, \ \ldots, \ L_{11}\},$$ where $\P(\s \rtimes (\z_2^\phi \times \z_2^\psi)) = \langle u_1, u_2, u_3, u_4, u_5 \rangle.$
\item[Type B:] When $a_0 = 1$ and $a_1^{n_2}a_2^{n_1} = -1,$ a set of generators for $\QQ_\sigma(\s \rtimes (\z_2^\phi \times \z_2^\psi))$ over $\P(\s \rtimes (\z_2^\phi \times \z_2^\psi))$ is 
$$\{L_0, \ u_4L_1, \ L_2, \ u_4L_3, \ u_4L_4, \ L_5, \ L_6, \ u_4L_7, \ u_4L_8, \ L_9, \ L_{10}, \ u_4L_{11}\},$$ where $\P(\s \rtimes (\z_2^\phi \times \z_2^\psi)) = \langle u_1, u_2, u_3, u_4^2, u_5 \rangle.$
\item[Type C:] When $a_0 = -1$ and $a_1^{n_2}a_2^{n_1} = 1,$ a set of generators for $\QQ_\sigma(\s \rtimes (\z_2^\phi \times \z_2^\psi))$ over $\P(\s \rtimes (\z_2^\phi \times \z_2^\psi))$ is 
$$\{u_1L_0, \ L_1, \ \ldots, \ L_{11}\},$$ where $\P(\s \rtimes (\z_2^\phi \times \z_2^\psi)) = \langle u_1^2, u_2, u_3, u_4, u_5 \rangle.$
\item[Type D:] When $a_0 = -1$ and $a_1^{n_2}a_2^{n_1} = -1,$ a set of generators for $\QQ_\sigma(\s \rtimes (\z_2^\phi \times \z_2^\psi))$ over $\P(\s \rtimes (\z_2^\phi \times \z_2^\psi))$ is 
$$\{u_1L_0, \ u_4L_0, \ u_1L_1, \ u_4L_1, \ L_2, \ u_1L_3, \ u_4L_3,\ u_1L_4, \ u_4L_4,$$ $$L_5, \ L_6, \ u_1L_7, \ u_4L_7, \ u_1L_8, \ u_4L_8,\ L_9, \ L_{10}, \ u_1L_{11},\ u_4L_{11}\},$$ where $\P(\s \rtimes (\z_2^\phi \times \z_2^\psi)) = \langle u_1^2, u_2, u_3, u_4^2, u_1u_4, u_5 \rangle.$
\end{itemize}

\begin{small}
\begin{table}[h]
	\centering
	\begin{tabular}{|l|c|c|}
		\hline
		Type & $\QQ_\sigma(\s \rtimes (\z_2^\phi \times \z_2^\psi))$ & $\P(\s \rtimes (\z_2^\phi \times \z_2^\psi))$\\
		\hline
		 A &   $L_j$, $~0 \leq j \leq 11$                     &   $u_1, u_2, u_3, u_4,u_5$\\
		\hline
		 B  &   $L_k$, $~u_4 L_\ell~$                            &  $u_1,u_2,u_3,u_4^2,u_5$ \\
		&   for $k=0,2,5,6,9,10 ~$ and $~\ell=1,3,4,7,8,11$   &                           \\
		\hline
		 C  &  $u_1H_0$, $~H_k~$ for $1 \leq k \leq 11$      & $u_1^2, u_2,u_3,u_4,u_5$\\
		\hline
		 D  &  $L_k$, $~u_1L_l$,  $~u_4 L_\ell~$                 & $u_1^2, u_2,u_3,u_4^2,u_1u_4,u_5$    \\
		&  for $~k=2,5,6,9,10~$ and $~\ell=0,1,3,4,7,8,11$ &                                          \\
		\hline
	\end{tabular}
\label{table-R2C3}
\vspace{0.2cm}
\caption{Generators on $\R^2 \times \C^3$ (see steps 1 and 3 of the algorithm in this subsection for the notation used in this table).}
\end{table}
\end{small}

Therefore, we have:

\begin{theorem}  Let $\dot{x}= X(x)$ be a {\rm $\z_2^\phi \times \z_2^\psi-$}reversible-equivariant system, with $L=(dX)_0$ defined as \eqref{L} for $n=3$ and $(n_1:n_2:0)-$resonant. Then this system is formally conjugate to one of the following:\\
	
\noindent Type A:
\begin{footnotesize}
\begin{align}
\dot{x}_1 &= x_2  +  x_1\im(z_1^{n_2}\bar{z}_2^{n_1})f_{0}(X),\nonumber\\
\dot{x}_2 &= f_{1}(X)   +    x_2\im(z_1^{n_2}\bar{z}_2^{n_1})f_{0}(X),\nonumber\\
\dot{z}_1 &= -i\omega_1z_1  +  iz_1f_{2}(X)  +   i\bar{z}_1^{n_2-1}z_2^{n_2}f_{3}(X)  +  z_1\im(z_1^{n_2}\bar{z}_2^{n_1})f_{4}(X)  + \bar{z}_1^{n_2-1}z_2^{n_2}\im(z_1^{n_2}\bar{z}_2^{n_1})f_{5}(X),\nonumber\\
\dot{z}_2 &= -i\omega_2z_2  +  iz_2f_{6}(X)   +  iz_1^{n_2}\bar{z}_2^{n_1-1}f_{7}(X)  +  z_2\im(z_1^{n_2}\bar{z}_2^{n_1})f_{8}(X)  + z_1^{n_2}\bar{z}_2^{n_1-1}\im(z_1^{n_2}\bar{z}_2^{n_1})f_{9}(X),  \nonumber\\
\dot{z}_3&= -i\omega_3z_3  +  i z_3f_{10}(X) +  z_3\im(z_1^{n_2}\bar{z}_2^{n_1}) f_{11}(X), \nonumber
\end{align}
\end{footnotesize} 
\noindent for  $f_i: \R^5,0 \to \R$, $i= 0, \ldots, 11$, and $X=(x_1$, $|z_1|^2$, $|z_2|^2$, $\re(z_1^{n_2}\bar{z}_2^{n_1})$,$|z_3|^2) $.\\

\noindent Type B:
\begin{footnotesize}
\begin{align}
\dot{x}_1 &= x_2 ~ +~  x_1 \re(z_1^{n_2}\bar{z}_2^{n_1}) \im(z_1^{n_2}\bar{z}_2^{n_1})f_{0}(X),\nonumber\\
\dot{x}_2 &= f_{1}(X)  ~ + ~   x_2 \re(z_1^{n_2}\bar{z}_2^{n_1}) \im(z_1^{n_2}\bar{z}_2^{n_1})f_{0}(X),\nonumber\\
\dot{z}_1 &= -i\omega_1z_1 ~ + ~ iz_1f_{2}(X)  + \bar{z}_1^{n_2-1}z_2^{n_2}\im(z_1^{n_2}\bar{z}_2^{n_1})f_{3}(X) ~+~  i\bar{z}_1^{n_2-1}z_2^{n_2}\re(z_1^{n_2}\bar{z}_2^{n_1})f_{4}(X) ~  + \nonumber,\\
&+ ~z_1\re(z_1^{n_2}\bar{z}_2^{n_1})\im(z_1^{n_2}\bar{z}_2^{n_1})f_{5}(X) \nonumber,\\
\dot{z}_2 &= -i\omega_2z_2 ~ + ~ iz_2f_{6}(X) ~ +~ z_1^{n_2}\bar{z}_2^{n_1-1}\im(z_1^{n_2}\bar{z}_2^{n_1})f_{7}(X) ~+~  i z_1^{n_2}\bar{z}_2^{n_1-1} \re(z_1^{n_2}\bar{z}_2^{n_1})f_{8}(X) ~ + ,\nonumber\\
&+ ~z_2\re(z_1^{n_2}\bar{z}_2^{n_1})\im(z_1^{n_2}\bar{z}_2^{n_1})f_{9}(X),  \nonumber\\
\dot{z}_3&= -i\omega_3z_3 ~+~ i z_3f_{10}(X) ~+~    z_3\re(z_1^{n_2}\bar{z}_2^{n_1})\im(z_1^{n_2}\bar{z}_2^{n_1}) f_{11}(X),    \nonumber
\end{align}
\end{footnotesize}
\noindent for  $f_i: \R^{5},0 \to \R$, $i= 0, \ldots, 11$, and  $X=(x_1$, $|z_1|^2$, $|z_2|^2$,$\re^2(z_1^{n_2}\bar{z}_2^{n_1})$, $|z_3|^2)$.\\

\noindent Type C:
\begin{footnotesize}
\begin{align}
\dot{x}_1 &= x_2  +  x_1\im(z_1^{n_2}\bar{z}_2^{n_1})f_{0}(X),\nonumber\\
\dot{x}_2 &= x_1f_{1}(X)   +    x_2\im(z_1^{n_2}\bar{z}_2^{n_1})f_{0}(X),\nonumber\\
\dot{z}_1 &= -i\omega_1z_1  +  iz_1f_{2}(X)  +   i\bar{z}_1^{n_2-1}z_2^{n_2}f_{3}(X)  +  z_1\im(z_1^{n_2}\bar{z}_2^{n_1})f_{4}(X)  + \bar{z}_1^{n_2-1}z_2^{n_2}\im(z_1^{n_2}\bar{z}_2^{n_1})f_{5}(X),\nonumber\\
\dot{z}_2 &= -i\omega_2z_2  +  iz_2f_{6}(X)   +  iz_1^{n_2}\bar{z}_2^{n_1-1}f_{7}(X)  +  z_2\im(z_1^{n_2}\bar{z}_2^{n_1})f_{8}(X)  + z_1^{n_2}\bar{z}_2^{n_1-1}\im(z_1^{n_2}\bar{z}_2^{n_1})f_{9}(X)  \nonumber\\
\dot{z}_3&= -i\omega_3z_3   +    iz_3f_{10}(X)   +   z_3\im(z_1^{n_2}\bar{z}_2^{n_1}) f_{11}(X),\nonumber
\end{align}
\end{footnotesize}
\noindent for  $f_i: \R^{5},0 \to \R$, $i= 0, \ldots, 11$, and $X=(x_1^2$, $|z_1|^2$, $|z_2|^2$, $\re(z_1^{n_2}\bar{z}_2^{n_1})$, $|z_3|^2)$.\\

\noindent Type D:
\begin{footnotesize}
\begin{align}
\dot{x}_1 &= x_2  +  x_1^2\im(z_1^{n_2}\bar{z}_2^{n_1})f_{0}(X)  +   x_1 \re(z_1^{n_2}\bar{z}_2^{n_1})\im(z_1^{n_2}\bar{z}_2^{n_1}) f_{1}(X), \nonumber\\
\dot{x}_2 &= x_1f_{2}(X)   +    x_1 x_2\im(z_1^{n_2}\bar{z}_2^{n_1})f_{0}(X) +  \re(z_1^{n_2}\bar{z}_2^{n_1})f_{3}(X)  +    x_2\re(z_1^{n_2}\bar{z}_2^{n_1}))\im(z_1^{n_2}\bar{z}_2^{n_1})f_{1}(X),     \nonumber\\
\dot{z}_1 &= -i\omega_1z_1  +  iz_1f_{4}(X)  + \bar{z}_1^{n_2-1}z_2^{n_2}\im(z_1^{n_2}\bar{z}_2^{n_1})f_{5}(X) +   ix_1\bar{z}_1^{n_2-1}z_2^{n_2}f_{6}(X) + \nonumber \\
&  +  x_1 z_1\im(z_1^{n_2}\bar{z}_2^{n_1})f_{7}(X)+ i\bar{z}_1^{n_2-1}z_2^{n_2} \re(z_1^{n_2}\bar{z}_2^{n_1})) f_{8}(X) +  z_1 \re(z_1^{n_2}\bar{z}_2^{n_1}) \im(z_1^{n_2}\bar{z}_2^{n_1})f_{9}(X), \nonumber\\
\dot{z}_2 &= -i\omega_2z_2  +  iz_2f_{10}(X)  + z_1^{n_2}\bar{z}_2^{n_1-1}\im(z_1^{n_2}\bar{z}_2^{n_1})f_{11}(X) +  ix_1z_1^{n_2}\bar{z}_2^{n_1-1}f_{12}(X) +  \nonumber \\
&  + x_1 z_2\im(z_1^{n_2}\bar{z}_2^{n_1})f_{13}(X)+ i z_1^{n_2}\bar{z}_2^{n_1-1}\re(z_1^{n_2}\bar{z}_2^{n_1}) f_{14}(X)+  z_2 \re(z_1^{n_2}\bar{z}_2^{n_1}) \im(z_1^{n_2}\bar{z}_2^{n_1}) f_{15}(X),   \nonumber\\
\dot{z}_3&= -i\omega_3z_3  +  iz_3 f_{16}(X) +  x_1 z_3\im(z_1^{n_2}\bar{z}_2^{n_1})f_{17}(X) + z_3\re(z_1^{n_2}\bar{z}_2^{n_1})\im(z_1^{n_2}\bar{z}_2^{n_1})f_{18}(X), \nonumber
\end{align}
\end{footnotesize}
\noindent for  $f_i: \R^{5},0 \to \R$, $i= 0, \ldots, 18$,   and $X=(x_1^2$, $|z_1|^2$, $|z_2|^2$, $\re^2(z_1^{n_2}\bar{z}_2^{n_1})$,   $x_1\re(z_1^{n_2}\bar{z}_2^{n_1})$, $|z_3|^2)$.
\end{theorem}

We remark that the value of $a_3$ has no effect on the normal forms. This is because $a_3$ in $\psi$ is acting on the algebraically independent part of $L.$ Hence, the results of this subsection  generalize to systems on $\R^2 \times \C^n$, $n>3$, which is done in the next subsection.

\subsection{Resonance of type $(n_1: n_2: 0)$ in $\R^2 \times \C^n$}

Here we extend the previous case to $n > 3$. 

The action of $\z_2^\phi \times \z_2^\psi$ on $\R^2 \times \C^n$ is given by \eqref{phiaction}, \eqref{psiactiononC}. We assume $n_1 \omega_2 - n_2 \omega_1 = 0,$ with $n_1, n_2 \in \mathbb{N}$ nonzero, and $\omega_3, \ldots, \omega_n$  algebraically independent. In this case,  the system (\ref{eq:SYSTEM})  is also called $(n_1:n_2:0)-$resonant. 

By Proposition~\ref{prop:SFORM}, $\s = \r \times \T^{n-1}.$ The diagonal action of $\r \times \T^{n-1}$ on $\R^2 \times \C^n$ is given from the standard action of $\r$ on $\R^2$ as in (\ref{S-actiononR2}) and the action of $\T^{n-1}$ on $\C^n$ is given by
\begin{equation*}
(\theta_1, \ldots, \theta_{n-1}) (z_1,z_2, \ldots, z_n)= (e^{in_1\theta_1}z_1, e^{in_2\theta_1}z_2, e^{i\theta_2}z_3, \ldots, e^{i\theta_{n-1}}z_n), \end{equation*}

\noindent with $(\theta_1, \ldots, \theta_{n-1}) \in \T^{n-1}.$

In this case, for the epimorphisms $\sigma_1$ and $\sigma_2$ given in (\ref{eq:EPI}) and $\sigma$ given in \eqref{sigma}, we obtain also four types of normal forms  (see Table \ref{table-typeNF}) with generators for $\QQ_\sigma( \s \rtimes (\z^\phi \times \z_2^\psi))$ and $\P( \s \rtimes (\z^\phi \times \z_2^\psi))$ given in Table 
\ref{table-R2Cn}, where:\\
\begin{small}
$ u_1(x,z)=x_1, \quad u_2(x,z)=|z_1|^2, \quad u_3(x,z)=|z_2|^2, \quad u_4(x,z)=\re(z_1^{n_2}\bar{z}_2^{n_1}),$

 $u_5=\im(z_1^{n_2}\bar{z}_2^{n_1}) , \quad u_6(x,z)= |z_3|^2, \quad u_7(x,z)= |z_4|^2,  \ldots,  u_{n+3}(x,z)= |z_n|^2$,

$H_0(x,z)=\begin{pmatrix} 0,1,0,\ldots,  0\end{pmatrix}$,  \hspace{0.2cm}
$H_1(x,z)=\begin{pmatrix} x_1\im(z_1^{n_2}\bar{z}_2^{n_1}), x_2\im(z_1^{n_2}\bar{z}_2^{n_1}) ,0,\ldots,  0 \end{pmatrix}$, 

$H_2(x,z)= \begin{pmatrix} 0,0,iz_1, 0, \ldots, 0 \end{pmatrix}$, \hspace{0.2cm}
$H_3(x,z)= \begin{pmatrix} 0,0, i\bar{z}_1^{n_2-1} z_2^{n_1},0,\ldots, 0 \end{pmatrix}$, 

$H_4(x,z)=\begin{pmatrix} 0,0, z_1 \im(z_1^{n_2}\bar{z}_2^{n_1}) ,0,\ldots, 0 \end{pmatrix}$,

$H_5(x,z)=\begin{pmatrix} 0,0,\bar{z}_1^{n_2-1} z_2^{n_1}\im(z_1^{n_2}\bar{z}_2^{n_1}) ,0,\ldots, 0 \end{pmatrix}$,  ~
$H_6(x,z)=\begin{pmatrix} 0,0,0 , iz_2,0, \ldots, 0\end{pmatrix}$, 

$H_7(x,z)=\begin{pmatrix} 0,0,0 , iz_1^{n_2}\bar{z}_2^{n_1-1},0, \ldots, 0\end{pmatrix}$,  \hspace{0.2cm}
$H_8(x,z)=\begin{pmatrix} 0,0,0 , z_2\im(z_1^{n_2}\bar{z}_2^{n_1}) ,0, \ldots, 0\end{pmatrix}$, 

$H_9(x,z)=\begin{pmatrix} 0,0,0, z_1^{n_2}\bar{z}_2^{n_1-1} \im(z_1^{n_2}\bar{z}_2^{n_1}) ,0,\ldots, 0\end{pmatrix}$, ~~
$H_{10}(x,z)= \begin{pmatrix} 0,0,0,0,iz_3,0, \ldots, 0 \end{pmatrix}$, 

$H_{11}(x,z)=\begin{pmatrix} 0,0,0,0, z_3 \im(z_1^{n_2}\bar{z}_2^{n_1}),0,\ldots,0 \end{pmatrix}$,\hspace{0.2cm}
$H_{12}(x,z)= \begin{pmatrix} 0,0,0,0,0,iz_4, \ldots, 0 \end{pmatrix}$, 

$H_{13}(x,z)=\begin{pmatrix} 0,0,0,0, 0,z_4 \im(z_1^{n_2}\bar{z}_2^{n_1}),0,\ldots,0 \end{pmatrix}$, ~
$\hdots$~, 

$H_{2n+4}(x,z) =  \begin{pmatrix} 0,0,0,0, \ldots, 0,i z_n \end{pmatrix}$, \hspace{0.2cm} 
$H_{2n+5}(x,z)=\begin{pmatrix} 0,0,0,0, \ldots,0,z_n \im(z_1^{n_2}\bar{z}_2^{n_1}) \end{pmatrix}$.
\end{small}

\begin{small} 
	\begin{table}[h]
		\centering
		\begin{tabular}{|c|l|c|c|}
			\hline
			\multirow{7}{*}{$a_0=1$} & $a_1=a_2=1$          & ---          & \multirow{2}{*}{Type A} \\ \cline{2-3}
			&\multirow{2}{*}{$a_1=a_2=-1$}& $n_1 + n_2$ even&              \\ \cline{3-4} 
			&                       & $n_1 + n_2$ odd & Type B       \\ \cline{2-4} 
			&\multirow{2}{*}{$a_1=-a_2=1$}& $n_1$ even      & Type A       \\ \cline{3-4} 
			&                       & $n_1$ odd       & Type B       \\ \cline{2-4} 
			& \multirow{2}{*}{$a_1=-a_2 =-1$}& $n_2$ even      & Type A       \\ \cline{3-4} 
			&                       & $n_2$ odd       & Type B       \\ \hline
			\multirow{7}{*}{$a_0=-1$}& $a_1=a_2=1$           & ---          & \multirow{2}{*}{Type C}\\ \cline{2-3}
			&\multirow{2}{*}{$a_1=a_2=-1$}& $n_1 + n_2$ even&              \\ \cline{3-4} 
			&                       & $n_1 + n_2$ odd & Type D       \\ \cline{2-4} 
			&\multirow{2}{*}{$ a_1=-a_2=1$}& $n_1$ even      & Type C       \\ \cline{3-4} 
			&                       & $n_1$ odd       & Type D       \\ \cline{2-4} 
			& \multirow{2}{*}{$a_1=-a_2=-1$}& $n_2$ even      & Type C       \\ \cline{3-4} 
			&                       & $n_2$ odd       & Type D       \\ \hline
		\end{tabular}
		\vspace{0.1cm}
		\caption{Types of $\z_2 \times \z_2-$reversible-equivariant vector fields.}
		\label{table-typeNF}
	\end{table}
\end{small}

\begin{small} 
	\begin{table}[h]
		\centering
		\begin{tabular}{|c|c|c|}
			\hline
			Type & $\QQ( \s \rtimes (\z^\phi \times \z_2^\psi))$ & $\P( \s \rtimes (\z^\phi \times \z_2^\psi))$\\
			\hline
			A &   $H_j$, $~0 \leq j \leq 2n+5$     &   $u_1, u_2, u_3, u_4,$\\
			&                                   &   $u_6, u_7, \ldots, u_{n+3}$    \\
			\hline
			B  &   $H_k$, $u_4 H_\ell~$ for $k=0,2,5,6,9,10,12,14, \ldots, 2n+4 $     &  $u_1,u_2,u_3,u_4^2,$ \\
			& and $\ell=1,3,4,7,8,11,13,15, \ldots, 2n+5$                             &  $u_6, u_7, \ldots, u_{n+3}$\\
			\hline
			C  &  $u_1H_0$, $~H_k~$ for $1 \leq k \leq 2n+5$     & $u_1^2, u_2,u_3,u_4,$ \\
			&                                                  &   $u_6, u_7, \ldots, u_{n+3} $\\
			\hline
			D  &  $H_k$, $~~u_1H_\ell$,  $~u_4 H_\ell~$ for $k=2,5,6,9,10,12,14, \ldots, 2n+4$  &  $u_1^2, u_2,u_3,u_4^2,u_1u_4,$ \\
			&  and $\ell=0,1,3,4,7,8,11,13,15, \ldots, 2n+5$                             &   $u_6, u_7, \ldots, u_{n+3}$\\
			\hline
		\end{tabular}
		\vspace{0.1cm}
		\caption{Generators on $\R^2 \times \C^n$.}
		\label{table-R2Cn}
	\end{table}
\end{small}

Therefore, we have the following result:

\begin{theorem}  Let $\dot{x}=X(x)$ be a  {\rm $\z_2^\phi \times \z_2^\psi-$}reversible-equivariant system, with $L=(dX)_0$ defined in (\ref{L}) for $n > 3$  and $(n_1:n_2:0)$-resonant. Then this system is formally conjugate to one of the following:\\

\noindent Type A:
\begin{footnotesize}
\begin{align}
\dot{x}_1 &= x_2  +  x_1\im(z_1^{n_2}\bar{z}_2^{n_1})f_{0}(X),\nonumber\\
\dot{x}_2 &= f_{1}(X)   +    x_2\im(z_1^{n_2}\bar{z}_2^{n_1})f_{0}(X),\nonumber\\
\dot{z}_1 &= -i\omega_1z_1 ~ + ~ iz_1f_{2}(X)  ~+ ~  i\bar{z}_1^{n_2-1}z_2^{n_2}f_{3}(X)  ~+ ~ z_1\im(z_1^{n_2}\bar{z}_2^{n_1})f_{4}(X) ~ +  \nonumber \\ &+~\bar{z}_1^{n_2-1}z_2^{n_2}\im(z_1^{n_2}\bar{z}_2^{n_1})f_{5}(X),\nonumber\\
\dot{z}_2 &= -i\omega_2z_2  ~+~  iz_2f_{6}(X) ~  + ~ iz_1^{n_2}\bar{z}_2^{n_1-1}f_{7}(X)  ~+ ~ z_2\im(z_1^{n_2}\bar{z}_2^{n_1})f_{8}(X) ~+ \nonumber \\ 
&+ ~z_1^{n_2}\bar{z}_2^{n_1-1}\im(z_1^{n_2}\bar{z}_2^{n_1})f_{9}(X),  \nonumber\\
\dot{z}_3&= -i\omega_3z_3 ~ + ~ i z_3f_{10}(X) ~+ ~ z_3\im(z_1^{n_2}\bar{z}_2^{n_1}) f_{11}(X),  \nonumber \\
\vdots & \nonumber\\
\dot{z}_n&= -i\omega_nz_n ~ + ~ i z_nf_{2n+4}(X) ~+ ~ z_n\im(z_1^{n_2}\bar{z}_2^{n_1}) f_{2n+5}(X),  \nonumber 
\end{align}
\end{footnotesize}
\noindent for $f_i: \R^{n+2},0 \to \R$, $0 \leq i \leq  2n +5$,  and  $X=(x_1$, $|z_1|^2$, $|z_2|^2$, $\re(z_1^{n_2}\bar{z}_2^{n_1})$, $|z_3|^2$, $\ldots$, $|z_n|^2)$.\\

\noindent Type B:
\begin{footnotesize}
\begin{align}
\dot{x}_1 &= x_2  +  x_1 \re(z_1^{n_2}\bar{z}_2^{n_1}) \im(z_1^{n_2}\bar{z}_2^{n_1})f_{0}(X),\nonumber\\
\dot{x}_2 &= f_{1}(X)   +    x_2 \re(z_1^{n_2}\bar{z}_2^{n_1}) \im(z_1^{n_2}\bar{z}_2^{n_1})f_{0}(X),\nonumber\\
\dot{z}_1 &= -i\omega_1z_1  +  iz_1f_{2}(X)  + \bar{z}_1^{n_2-1}z_2^{n_2}\im(z_1^{n_2}\bar{z}_2^{n_1})f_{3}(X) +  i\bar{z}_1^{n_2-1}z_2^{n_2}\re(z_1^{n_2}\bar{z}_2^{n_1})f_{4}(X) ~  +  \nonumber\\
& +  z_1\re(z_1^{n_2}\bar{z}_2^{n_1})\im(z_1^{n_2}\bar{z}_2^{n_1})f_{5}(X), \nonumber\\
\dot{z}_2 &= -i\omega_2z_2  +  iz_2f_{6}(X)  + z_1^{n_2}\bar{z}_2^{n_1-1}\im(z_1^{n_2}\bar{z}_2^{n_1})f_{7}(X) +  i z_1^{n_2}\bar{z}_2^{n_1-1} \re(z_1^{n_2}\bar{z}_2^{n_1})f_{8}(X) ~ + \nonumber\\
&+  z_2\re(z_1^{n_2}\bar{z}_2^{n_1})\im(z_1^{n_2}\bar{z}_2^{n_1})f_{9}(X),   \nonumber\\
\dot{z}_3&= -i\omega_3z_3 + i z_3f_{10}(X) +    z_3\re(z_1^{n_2}\bar{z}_2^{n_1})\im(z_1^{n_2}\bar{z}_2^{n_1}) f_{11}(X),    \nonumber\\
\vdots & \nonumber \\
\dot{z}_n&= -i\omega_nz_n + i z_nf_{2n+4}(X) +    z_n\re(z_1^{n_2}\bar{z}_2^{n_1})\im(z_1^{n_2}\bar{z}_2^{n_1}) f_{2n+5}(X),    \nonumber
\end{align}
\end{footnotesize}
\noindent for $f_i: \R^{n+2},0 \to \R$, $~0 \leq i \leq 2n+5 $,  and $X=(x_1$, $|z_1|^2$, $|z_2|^2$, $\re^2(z_1^{n_2}\bar{z}_2^{n_1})$, $|z_3|^2$, $\ldots$,~$|z_n|^2)$.\\

\noindent Type C:
\begin{footnotesize}
\begin{align}
\dot{x}_1 &= x_2  ~+ ~ x_1\im(z_1^{n_2}\bar{z}_2^{n_1})f_{0}(X),\nonumber\\
\dot{x}_2 &= x_1f_{1}(X)  ~ +  ~  x_2\im(z_1^{n_2}\bar{z}_2^{n_1})f_{0}(X),\nonumber\\
\dot{z}_1 &= -i\omega_1z_1  ~+ ~ iz_1f_{2}(X) ~ + ~  i\bar{z}_1^{n_2-1}z_2^{n_2}f_{3}(X) ~ +  ~z_1\im(z_1^{n_2}\bar{z}_2^{n_1})f_{4}(X) ~ +\nonumber \\ 
&+ ~\bar{z}_1^{n_2-1}z_2^{n_2}\im(z_1^{n_2}\bar{z}_2^{n_1})f_{5}(X),\nonumber\\
\dot{z}_2 &= -i\omega_2z_2 ~ + ~ iz_2f_{6}(X)  ~ + ~ iz_1^{n_2}\bar{z}_2^{n_1-1}f_{7}(X)  ~+ ~ z_2\im(z_1^{n_2}\bar{z}_2^{n_1})f_{8}(X) ~ + \nonumber \\ &+~z_1^{n_2}\bar{z}_2^{n_1-1}\im(z_1^{n_2}\bar{z}_2^{n_1})f_{9}(X),  \nonumber\\
\dot{z}_3&= -i\omega_3z_3  ~ +  ~  iz_3f_{10}(X)   ~+ ~  z_3\im(z_1^{n_2}\bar{z}_2^{n_1}) f_{11}(X),\nonumber\\
\vdots & \nonumber\\
\dot{z}_3&= -i\omega_nz_n  ~ + ~   iz_nf_{2n+4}(X)  ~ +  ~ z_n\im(z_1^{n_2}\bar{z}_2^{n_1}) f_{2n+5}(X),\nonumber
\end{align}
\end{footnotesize}
\noindent for  $f_i: \R^{n+2},0 \to \R$, $~0 \leq i \leq 2n+5$,   $~X=(x_1^2$, $|z_1|^2$, $|z_2|^2$, $\re(z_1^{n_2}\bar{z}_2^{n_1})$, $|z_3|^2$, $\ldots$,~$|z_n|^2)$.\\

\noindent Type D:
\begin{footnotesize}
\begin{align}
\dot{x}_1 &= x_2  +  x_1^2\im(z_1^{n_2}\bar{z}_2^{n_1})f_{0}(X)  +   x_1 \re(z_1^{n_2}\bar{z}_2^{n_1}))\im(z_1^{n_2}\bar{z}_2^{n_1}) f_{1}(X) \nonumber\\
\dot{x}_2 &= x_1f_{2}(X)   +    x_1 x_2\im(z_1^{n_2}\bar{z}_2^{n_1})f_{0}(X) +  \re(z_1^{n_2}\bar{z}_2^{n_1}))f_{3}(X)  +    x_2\re(z_1^{n_2}\bar{z}_2^{n_1}))\im(z_1^{n_2}\bar{z}_2^{n_1})f_{1}(X),   \nonumber\\
\dot{z}_1 &= -i\omega_1z_1  +  iz_1f_{4}(X)  + \bar{z}_1^{n_2-1}z_2^{n_2}\im(z_1^{n_2}\bar{z}_2^{n_1})f_{5}(X) +   ix_1\bar{z}_1^{n_2-1}z_2^{n_2}f_{6}(X) + \nonumber \\
&  +  x_1 z_1\im(z_1^{n_2}\bar{z}_2^{n_1})f_{7}(X)+ i\bar{z}_1^{n_2-1}z_2^{n_2} \re(z_1^{n_2}\bar{z}_2^{n_1})) f_{8}(X) +  z_1 \re(z_1^{n_2}\bar{z}_2^{n_1})) \im(z_1^{n_2}\bar{z}_2^{n_1})f_{9}(X) \nonumber\\
\dot{z}_2 &= -i\omega_2z_2  +  iz_2f_{10}(X)  + z_1^{n_2}\bar{z}_2^{n_1-1}\im(z_1^{n_2}\bar{z}_2^{n_1})f_{11}(X) +  ix_1z_1^{n_2}\bar{z}_2^{n_1-1}f_{12}(X) +  \nonumber \\
& + x_1 z_2\im(z_1^{n_2}\bar{z}_2^{n_1})f_{13}(X)+ i z_1^{n_2}\bar{z}_2^{n_1-1}\re(z_1^{n_2}\bar{z}_2^{n_1}) f_{14}(X)+  z_2 \re(z_1^{n_2}\bar{z}_2^{n_1}) \im(z_1^{n_2}\bar{z}_2^{n_1}) f_{15}(X),   \nonumber\\
\dot{z}_3&= -i\omega_2z_3  +  iz_3 f_{16}(X) +  x_1 z_3\im(z_1^{n_2}\bar{z}_2^{n_1})f_{17}(X) + z_3\re(z_1^{n_2}\bar{z}_2^{n_1})\im(z_1^{n_2}\bar{z}_2^{n_1})f_{18}(X), \nonumber\\
\vdots& \nonumber\\
\dot{z}_n&= -i\omega_nz_n  +  iz_n f_{3n+7}(X) +  x_1 z_n\im(z_1^{n_2}\bar{z}_2^{n_1})f_{3n+8}(X) + z_n\re(z_1^{n_2}\bar{z}_2^{n_1})\im(z_1^{n_2}\bar{z}_2^{n_1})f_{3n+9}(X), \nonumber
\end{align}
\end{footnotesize}
for $f_i: \R^{n+3},0 \to \R$, $~0 \leq i \leq 3n+9$,  and $X=(x_1^2$, $|z_1|^2$, $|z_2|^2$, $\re^2(z_1^{n_2}\bar{z}_2^{n_1})$,  $ x_1\re(z_1^{n_2}\bar{z}_2^{n_1})$, $|z_3|^2$,  $\ldots$, $|z_n|^2)$.
\end{theorem}


\subsection{Resonance of type $(n_1: n_2 - m_1: m_2)$ in $\R^2 \times \C^4$} 

We assume $n_1 \omega_2 - n_2 \omega_1 = m_1 \omega_4 - m_2 \omega_3 = 0,$ with $n_1, n_2, m_1, m_2 \in \mathbb{N}$ nonzero. Under these conditions, the system (\ref{eq:SYSTEM}) is called $(n_1:n_2:m_1:m_2)$-resonant. 

By Proposition~\ref{prop:SFORM}, $\s = \r \times \T ^2$ and its action on $\R^2 \times \C^4$ is determined from the standard action of $\r$ on $\R^2$ given in (\ref{S-actiononR2}) and the diagonal action of $\T^2$  on $\C^4$ given by  
\begin{equation*}
(\theta_1, \theta_2) (z_1, z_2, z_3, z_4) = (e^{in_1\theta_1}z_1,e^{in_2\theta_1} z_2, e^{im_1\theta_2}z_3,e^{im_2\theta_2} z_4).
\end{equation*}



We follow the same steps as in the previous subsections, so here we shall omit the details.



The polynomial functions
$$
u_1(x,z)=x_1, ~u_2(x,z)= |z_1|^2,~u_3(x,z)= |z_2|^2, ~u_4(x,z)=\re(z_1^{n_2}\bar{z}_2^{n_1}), 
$$
$$
u_5(x,z)=\im(z_1^{n_2}\bar{z}_2^{n_1}),~u_6(x,z)=|z_3|^2,~u_7(x,z)=|z_4|^2,
$$
$$
u_8(x,z)=\re(z_3^{m_2}\bar{z}_4^{m_1}) \quad \text{and} \quad u_9(x,z)=\im(z_3^{m_2}\bar{z}_4^{m_1})
$$
form a Hilbert basis for $\P(\s)$.  \\

The generators of $\PP(\s)$ over $\P(\s)$ are:

$H_0(x,z)=\begin{pmatrix} x_1 , x_2 ,0,0,0,0 \end{pmatrix}$, \hspace{0.2cm}
$H_1(x,z)=\begin{pmatrix} 0 , 1 ,0,0,0,0\end{pmatrix}$, 

$H_2(x,z)=\begin{pmatrix} 0,0, z_1, 0 ,0,0  \end{pmatrix}$,  \hspace{0.2cm}
$H_3(x,z)=\begin{pmatrix} 0,0,iz_1, 0,0,0  \end{pmatrix}$, 

$H_4(x,z)=\begin{pmatrix} 0, 0, \bar{z}_1^{n_2-1} z_2^{n_1} ,0,0,0  \end{pmatrix}$,\hspace{0.2cm} 
$H_5(x,z)=\begin{pmatrix} 0,0,i\bar{z}_1^{n_2-1} z_2^{n_1} ,0,0,0  \end{pmatrix}$,

$H_6(x,z)=\begin{pmatrix} 0,0,0 , z_2,0,0 \end{pmatrix}$, \hspace{0.2cm} 
$H_7(x,z)=\begin{pmatrix} 0,0,0 , iz_2,0,0 \end{pmatrix}$, 
 
$H_8(x,z)=\begin{pmatrix} 0,0,0 , z_1^{n_2}\bar{z}_2^{n_1-1},0,0 \end{pmatrix}$,\hspace{0.2cm}  
$H_9(x,z)=\begin{pmatrix} 0,0,0, iz_1^{n_2}\bar{z}_2^{n_1-1},0,0 \end{pmatrix}$, 

$H_{10}(x,z)=\begin{pmatrix} 0,0,0,0,z_3,0 \end{pmatrix}$, \hspace{0.2cm}
$H_{11}(x,z)=\begin{pmatrix} 0,0,0,0,i z_3,0 \end{pmatrix}$.

$H_{12}(x,z)=\begin{pmatrix} 0, 0, 0,0,\bar{z}_3^{m_2-1} z_4^{m_1} ,0 \end{pmatrix}$,\hspace{0.2cm} 
$H_{13}(x,z)=\begin{pmatrix} 0,0,0,0,  i\bar{z}_3^{m_2-1} z_4^{m_1} ,0  \end{pmatrix}$,

$H_{14}(x,z)=\begin{pmatrix} 0,0,0,0,0 , z_4 \end{pmatrix}$, \hspace{0.2cm} 
$H_{15}(x,z)=\begin{pmatrix} 0,0,0,0,0 , iz_4 \end{pmatrix}$, 
 
$H_{16}(x,z)=\begin{pmatrix} 0,0,0,0 , z_3^{m_2}\bar{z}_4^{m_1-1} \end{pmatrix}$, \hspace{0.2cm}  
$H_{17}(x,z)=\begin{pmatrix} 0,0,0,0, iz_3^{m_2}\bar{z}_4^{m_1-1} \end{pmatrix}.$ 

\quad

\noindent Therefore, the generators of $\QQ_{\sigma_1}(\s \rtimes \z_2^\phi)$ over $\P(\s \rtimes \z_2^\phi)$ are:
$$
H_i(x,z), ~~ u_5(x,z) H_j(x,z) ~~ \text{e}~~ u_9(x,z)H_j(x,z),
$$
for $~i=1,3,5,7,9,11,13,15,17~$ e $~j=0,2,4,6,8,10,12,14,16$.

\noindent A Hilbert basis for $\P(\s \rtimes \z_2^\phi)$ is $\{v_1, \ldots, v_8\},$ where
\begin{equation*} 
v_1=u_1 , ~~v_2=u_2, ~~v_3=u_3,  ~~v_4=u_4, ~~ v_5=u_6,
\end{equation*}
$$
v_6=u_7 , ~~v_7=u_8~~ \text{e}~~v_8=u_5u_9,
$$ and a Hilbert basis for $\P(\s \rtimes (\z_2^\phi \times \z_2^\psi))$ is given by the polynomial functions:

$(1+ a_0)v_1$, $~~v_2$,$~~ v_3$, $ ~~(1 + a_1^{n_2} a_2^{n_1})  v_4$, $~~v_5$,$~~v_6$, $~~(1 + a_3^{m_2} a_4^{m_1}) v_7^2$ ,  
$(1 + a_1^{n_2} a_2^{n_1}a_3^{m_2} a_4^{m_1}) v_8$,  $~~ \displaystyle \hspace*{.4cm} (1 -  a_0)v_1^2$, $~~ (1 - a_1^{n_2} a_2^{n_1}) v_4^2$, $~~(1 - a_3^{m_2} a_4^{m_1}) v_7^2$,

$( 1 -  a_1^{n_2} a_2^{n_1}a_3^{m_2} a_4^{m_1}) v_8^2$,
$ (1-a_0) (1- a_1^{n_2} a_2^{n_1}) v_1 ~v_4$,   

$( 1-a_0)( 1- a_3^{m_2} a_4^{m_1}) v_1~v_7$,

$( 1-a_0)(1- a_1^{n_2} a_2^{n_1} a_3^{m_2} a_4^{m_1}) v_1~v_8$,

$( 1- a_1^{n_2} a_2^{n_1}) (1- a_3^{m_2} a_4^{m_1})v_4~v_7$,

$ (1- a_1^{n_2} a_2^{n_1})(1- a_1^{n_2} a_2^{n_1} a_3^{m_2} a_4^{m_1})v_4~v_8$,

$(1- a_3^{m_2} a_4^{m_1})(1- a_1^{n_2} a_2^{n_1} a_3^{m_2} a_4^{m_1})v_7 ~v_8$. \\

The generators for $\QQ_\sigma(\s \rtimes (\z_2^\phi \times \z_2^\psi))$ over $\P(\s \rtimes (\z_2^\phi \times \z_2^\psi))$ are:

$\tilde{J}_{01} (x,z)=(1 + a_0)H_1$, \hspace{0.2cm}
$\tilde{J}_{0i}  =  H_i$,
$\tilde{J}_{0k}=(1 + a_1^{n_2} a_2^{n_1})H_k$,

$\tilde{J}_{0\ell}=(1 + a_3^{m_2} a_4^{m_1})H_\ell$,
$\tilde{J}_{11}=(1 - a_0)v_1 ~H_1$, \hspace{0.2cm}
$\tilde{J}_{1i}  = 0$,

$\tilde{J}_{1k} =(1 - a_0)(1 -  a_1^{n_2} a_2^{n_1}) v_1 ~H_k$,
$\tilde{J}_{1\ell}=(1 - a_0)(1 - a_3^{m_2} a_4^{m_1}) v_1 ~H_l$,

$\tilde{J}_{41}=(1 - a_0) (1 -  a_1^{n_2} a_2^{n_1}) v_4( ~H_1$,\hspace{0.2cm}
$\tilde{J}_{4i}=0$,
$\tilde{J}_{4k} =(1 -  a_1^{n_2} a_2^{n_1}) v_4 ~H_k$,

$\tilde{J}_{4\ell}=(1 -  a_1^{n_2} a_2^{n_1}) (1 - a_3^{m_2} a_4^{m_1}) v_4 ~H_l$,
$\tilde{J}_{71} =(1 -  a_0) (1 - a_3^{m_2} a_4^{m_1}) v_7 ~H_1$, \hspace{0.2cm}

$\tilde{J}_{7i}=0$,
$\tilde{J}_{7k} =(1 -  a_1^{n_2} a_2^{n_1})(1 - a_3^{m_2} a_4^{m_1}) v_7 ~H_k$,

$\tilde{J}_{7\ell}=(1 - a_3^{m_2} a_4^{m_1}) v_7 ~H_\ell$, \hspace{0.2cm}
$\tilde{J}_{8i}=0$,
$\tilde{J}_{81}=(1 - a_0)(1 - a_1^{n_2} a_2^{n_1} a_3^{m_2} a_4^{m_1}) v_8 ~H_1$,

$\tilde{J}_{8k} =\displaystyle  (1 -  a_1^{n_2} a_2^{n_1})(1 - a_1^{n_2} a_2^{n_1} a_3^{m_2} a_4^{m_1}) v_8 ~H_k$, 

$\tilde{J}_{8\ell} =(1 -  a_3^{m_2} a_4^{m_1}) (1 - a_1^{n_2} a_2^{n_1} a_3^{m_2} a_4^{m_1}) v_8 ~H_l$,

$\tilde{J}_{05j} =(1 +  a_1^{n_2} a_2^{n_1}) u_5 ~H_j$, \hspace{0.2cm}
$\tilde{J}_{05r} = u_5 ~H_r$,
$\tilde{J}_{05s}=(1 +  a_1^{n_2} a_2^{n_1}a_3^{m_2} a_4^{m_1}) u_5 ~H_s$,

$\tilde{J}_{15j} =(1 - a_0)  (1 +  a_1^{n_2} a_2^{n_1}) v _1 ~u_5 ~H_s$,
$\tilde{J}_{15r}=(1 - a_0)  v_1\ ~u_5 ~H_r$,

$\tilde{J}_{15s}=(1 - a_0) (1 +  a_1^{n_2} a_2^{n_1}a_3^{m_2} a_4^{m_1}) v_1 ~u_5 ~H_s$,
$\tilde{J}_{45j}=(1 -  a_1^{n_2} a_2^{n_1}) v _4 ~u_5 ~H_j$, \hspace{0.2cm}

$\tilde{J}_{45r} =0$,
$\tilde{J}_{45s}=(1 - a_1^{n_2} a_2^{n_1})( 1 - a_1^{n_2} a_2^{n_1}a_3^{m_2} a_4^{m_1}) v_4 ~u_5 ~H_s$,

$\tilde{J}_{75j} =(1 - a_1^{n_2} a_2^{n_1})(1 - a_3^{m_2} a_4^{m_1}) u_5~ v_7 ~H_j(x,z)$, 
$\tilde{J}_{75r}=0$,

$\tilde{J}_{75s} =(1 - a_3^{m_2} a_4^{m_1}) (1 - a_1^{n_2} a_2^{n_1}a_3^{m_2} a_4^{m_1}) u_5~v_7 ~H_s$,

$\tilde{J}_{85j} =(1 - a_1^{n_2} a_2^{n_1}) (1 - a_1^{n_2} a_2^{n_1} a_3^{m_2} a_4^{m_1}) u_5~ v_8\ ~H_j$, 

$\tilde{J}_{85r} =0$, \hspace{0.2cm}
$\tilde{J}_{85s}=(1 - a_1^{n_2} a_2^{n_1}a_3^{m_2} a_4^{m_1}) u_5~v_8 ~H_s$,

$\tilde{J}_{09j} =(1 +  a_3^{m_2} a_4^{m_1}) u_9 ~H_j$, \hspace{0.2cm}
$\tilde{J}_{09s} =u_9 ~H_j$,

$\tilde{J}_{09r}=(1 +  a_1^{n_2} a_2^{n_1} a_3^{m_2} a_4^{m_1}) u_9 ~H_r$,
$\tilde{J}_{19j}=(1-a_0) (1 +  a_3^{m_2} a_4^{m_1})v_1~ u_9 ~H_j$,

$\tilde{J}_{19r}=(1-a_0)(1 +  a_1^{n_2} a_2^{n_1} a_3^{m_2} a_4^{m_1})v_1~ u_9 ~H_r$,
$\tilde{J}_{19s}=( 1-a_0)  v_1~ u_9 ~H_s$,

$\tilde{J}_{49j}=(1 - a_3^{m_2} a_4^{m_1})v_4~ u_9 ~H_j$,
$\tilde{J}_{49r}=(1 - a_1^{n_2} a_2^{n_1} a_3^{m_2} a_4^{m_1})v_4~ u_9 ~H_r$, \hspace{0.2cm}
$\tilde{J}_{49r} =0$,

$\tilde{J}_{79j}=( 1 - a_3^{m_2} a_4^{m_1}) v_7~u_9 ~H_j$,
$\tilde{J}_{79r}=(1 - a_3^{m_2} a_4^{m_1}) (1 - a_1^{n_2} a_2^{n_1} a_3^{m_2} a_4^{m_1})v_7~  u_9 ~H_r$,

$\tilde{J}_{79s} =0$,
$\tilde{J}_{89j}=(1 - a_3^{m_2} a_4^{m_1}) (1 - a_1^{n_2} a_2^{n_1}a_3^{m_2} a_4^{m_1}) v_8~u_9 ~H_j$,

$\tilde{J}_{89r} =(1 - a_1^{n_2} a_2^{n_1}a_3^{m_2} a_4^{m_1})v_8~ u_9 ~H_r$, \hspace{0.2cm}
$\tilde{J}_{89s} =0$, \\

\noindent for $~i=3,7,11,15$, $~j=0,2,6,10,14$, $~k=5,9$, $~\ell=13,17$, $~r=4,8~$ and $~s=12,16$.

\quad

Again, if $\phi=\psi$ then the data above give  the $\z_2-$reversible normal form. Also, 
if  the vector field of \eqref{eq:SYSTEM} has linearization at the origin $L = (dX)_0$ with only imaginary eigenvalues (no nilpotent part)  then, all the normal forms presented in this paper are rewritten in $\C^n$ for some $n$ to produce the normal forms in this case. In fact, omit the variables $x_1$ and $x_2$ and all the generators that depend only on these variables.
Finally, we notice that the pairs of involutions $(\phi,\psi)$ that anti-commute with $L$ have been considered by assuming  $\omega_i^2 \neq \omega_j^2$ for  $i, j = 1, \ldots , n, $ $i \neq j$. However, these pairs also anti-commute with $L$ if $\omega_i^2 = \omega_j^2$ for some $i, j$,  $i \neq j$. This case corresponds to the $1:1-$resonance, which is also of interest in dynamical systems. 



\end{document}